\documentclass[12pt]{article}

\usepackage{amsmath}
\usepackage{amsthm}
\usepackage{amssymb}
\usepackage[utf8]{inputenc}
\usepackage[english]{babel}
\usepackage{amsmath}
\usepackage{amssymb}
\usepackage{amsfonts}
\usepackage{mathrsfs}
\usepackage{graphicx}
\usepackage{floatflt}
\usepackage{mathtools}
\usepackage{amsthm}
\usepackage[T1]{fontenc}
\usepackage{tikz}
\usepackage{graphicx}
\graphicspath{ {./images/} }
\usepackage{MnSymbol}
\usepackage{relsize}
\usepackage{pst-solides3d,pgfplots}
\usepackage{pstricks}
\usepackage{tqft}
\usepackage{tikz-cd}
\usepackage{bm}
\usepackage{fancyhdr}
\usepackage{lipsum}
\usepackage{bm}
\usepackage{parskip}
\usepackage[hidelinks]{hyperref}
\usepackage{geometry}
\usepackage{amsthm}

\newtheorem{theorem}{Theorem}

\newtheorem{lemma}[theorem]{Lemma}

\newtheorem{proposition}[theorem]{Proposition}

\newcommand{\CC}{\mathbb{C}}

\newcommand{\RR}{\mathbb{R}}
\newcommand{\ZZ}{\mathbb{Z}}

\newcommand{\NN}{\mathbb{N}}

\usepackage[english]{babel}
\usepackage[utf8]{inputenc}
\usepackage{wasysym, stackengine, makebox, graphicx}

\usepackage{sectsty}
\sectionfont{\large}
\subsectionfont{\small}
\usepackage[font=small,labelfont=bf]{caption}
\newtheorem*{theorem*}{Theorem}
\newtheorem*{proposition*}{Proposition}
\usepackage{amsthm}
\newtheorem{alphateo}{Theorem}

\usepackage{layout}
\usepackage{geometry}
\usepackage{faktor}

\newlength\titlebox 
\title{Monodromy kernels for strata of translation surfaces}

\author{Riccardo Giannini}

\newcommand{\Addresses}{{
  \bigskip
  \footnotesize

  \par\nopagebreak
  \textit{E-mail address}, R.~Giannini: \texttt{r.giannini.1@research.gla.ac.uk}

  \medskip
}}
\date{} 

\begin{document}

\maketitle

\begin{abstract} 
\noindent 
    {\footnotesize The non-hyperelliptic connected components of the strata of translation surfaces are conjectured to be orbifold classifying spaces for some groups commensurable to some mapping class groups. The topological monodromy map of the non-hyperelliptic components projects naturally to the mapping class group of the underlying punctured surface and is an obvious candidate to test commensurability. In the present article, we prove that for the components $\mathcal{H}(3,1)$ and $\mathcal{H}^{nh}(4)$ in genus 3 the monodromy map fails to demonstrate the conjectured commensurability. In particular, building on work of Wajnryb, we prove that the kernels of the monodromy maps for $\mathcal{H}(3,1)$ and $\mathcal{H}^{nh}(4)$ are large, as they contain a non-abelian free group of rank $2$.} \end{abstract}

\section*{Introduction}

Translation surfaces and their moduli spaces naturally arise in the interplay of topology, algebraic geometry, dynamics and number theory as shown through the work of Veech \cite{Veech}, Masur \cite{Masur82}, Thurston \cite{Thurston} and many subsequent authors. The topology of the moduli spaces of Riemann surfaces is somewhat understood; see, for example, the work of Harer-Zagier \cite{HarerZagier} and Maclachlan \cite{Maclachlan}. Much less is known about the topology of the moduli space of translation surfaces.

In this article, we analyze the topological monodromy map of the connected components $\mathcal{H}^{nh}(4)$ and $\mathcal{H}(3,1)$ of the moduli space of translation surfaces in genus $3$. We show that the topological monodromy maps from the orbifold fundamental groups onto the images in the respective mapping class groups are far from being isomorphisms. In particular, we prove that the kernels contain a non-abelian free group of rank $2$ by relating the topological monodromy maps to some geometric homomorphisms of Artin groups.  

\textbf{Translation surfaces.} Let $\Sigma_g$ denote a closed oriented surface of genus $g$ and let $\mathcal{Z}\subset\Sigma_g$ be a finite set of points. A \textit{translation structure} on $\Sigma_g$ is an atlas of charts with values in $\CC$ where the transition maps of $\Sigma_g\setminus\mathcal{Z}$ are translation, points in $\mathcal{Z}$ are cone type singularities and the holonomy $\pi_1(\Sigma_g\setminus\mathcal{Z})\rightarrow \operatorname{SO}(2)$ is trivial. In particular, the complex structure on $\Sigma_g\setminus\mathcal{Z}$ can be extended to $\Sigma_g$ by Riemann's removable singularity theorem and the metric around each point $p\in\mathcal{Z}$ can be given by cyclically gluing half-planes around $p$.

A translation structure on $\Sigma_g$ can also be given by pairs of the form $(X,\omega)$, where $X$ is genus $g$ Riemann surface and $\omega$ is a non-zero holomorphic one form on $X$. The finite set $\mathcal{Z}$ is identified with $\mathcal{Z}(\omega)=\{p\in X\mid\omega_p\equiv0\}$. Since the holonomy around every cone singularity is trivial, the number $k_p$ of half-planes glued around each point $p\in\mathcal{Z}$ is even. The multiplicity of $\omega$ at the respective vanishing point is exactly $\frac{k_p}{2}+1$.

\textbf{Strata of translation surfaces.} The moduli space of genus $g$ translation surfaces is the set of all translation structures $(X,\omega)$ of $\Sigma_g$ up to isomorphisms. The whole moduli space can be stratified in orbifolds $\mathcal{H}(k_1,\dots,k_n)$ characterized by the combinatorial data given by the orders of $\omega$ at its zeros. 

Even though the topology of the strata of translation surfaces is poorly understood, our knowledge has improved in the past years. Costantini-M\"{o}ller-Zoachhuber gave a recursive computable formula for the Euler characteristic of the moduli space of translation surfaces \cite{Costantini2022}. Further, Zykoski has constructed a finite simplicial complex with the same homotopy type of the strata $\mathcal{H}(k_1,\dots,k_n)$, motivated by Harer's construction of a simplicial complex those quotient by the mapping class group is homotopic equivalent to the moduli space of Riemann surfaces \cite{zykoski2022isodelaunay}. 

Kontsevich-Zorich showed that each stratum has at most $3$ connected components and in every genus some components are \textit{hyperelliptic} \cite{Kontsevich2003}. Namely, hyperelliptic components consist of translation surfaces $(X,\omega)$ where $X$ is a hyperelliptic Riemann surface and $\iota^*(X,\omega)=(X,-\omega)$ where $\iota$ is the hyperelliptic involution of $X$. These connected components are orbifold classifying spaces for finite extensions of braid groups; see \cite[Section 1.4]{Looijenga2014} for a proof. Kontsevich-Zorich also conjectured that the other non-hyperelliptic components have orbifold fundamental groups commensurable with some mapping class group \cite{Kontsevich1997}; this conjecture is still open.

Our focus is to shed some light on the Kontsevich-Zorich conjecture for some non-hyperelliptic components in small complexity, by showing that the topological monodromy maps of some exceptional connected components in the mapping class group are far from being injective.

\textbf{Monodromy maps.} Let $\Sigma_{g,n}$ be a closed surface with $n$ marked points. The mapping class group $\operatorname{Mod}_{g,n}$ is the group of all orientation preserving self-diffeomorphism of $\Sigma_{g,n}$ that leave the set of marked points invariant, up to isotopies relative to the set of marked points. If $\mathcal{C}$ is a connected component of a stratum $\mathcal{H}(k_1,\dots,k_n)$, then any (orbifold) homotopy class of loops based at $(X,\omega)$ gives rise to some self-diffeomorphism of $X$ that preserves the zeros of $\omega$. These data are recorded by the (punctured) topological monodromy map:
$$\rho_\mathcal{C}:\pi_1^{orb}(\mathcal{C})\rightarrow\operatorname{Mod}_{g,n}.$$

Calderon studied these homomorphisms and described the connected components of the strata of \textit{marked translation surfaces} for genus $g\geq 5$, which cover the strata of translation surfaces  \cite{CalderonConnected2020}.  Then, Calderon-Salter's work resulted in a complete description of the images of the monodromy maps associated with all non-hyperelliptic connected components of the strata $\mathcal{H}(k_1,\dots,k_n)$ in genus $g\geq 5$. In other words, the orbifold fundamental groups of all non-hyperelliptic connected components in genus $g\geq 5$ are projected onto subgroups of the mapping class group called \textit{framed mapping class groups} \cite{CalderonSalterFramed2022}. 

\textbf{The kernel of the punctured monodromy.} If $\mathcal{C}$ is hyperelliptic then $\operatorname{Im}\rho_\mathcal{C}$ is isomorphic to the symmetric mapping class group $\operatorname{SMod}_{g,n}$, and $\ker\rho_\mathcal{C}$ is finite \cite[Section 2.1]{CalderonConnected2020}. In view of Kontsevich-Zorich conjecture, it is natural to ask whether or not the topological monodromy is the right homomorphism to look at in order to prove the conjecture. For this reason, we are interested in estimating the size of the kernels of the monodromies $\rho_\mathcal{C}$ for non-hyperelliptic connected components. The first main result of this paper is that in some cases the kernel is large.

\begin{alphateo}
\label{ker}
    Let $\rho_{\mathcal{H}^{nh}(4)}:\pi_1^{orb}(\mathcal{H}^{nh}(4))\rightarrow\operatorname{Mod}_{3,1}$ and $\rho_{\mathcal{H}(3,1)}:\pi_1^{orb}(\mathcal{H}(3,1))\rightarrow\operatorname{Mod}_{3,2}$ be the topological monodromy maps of the non-hyperelliptic connected components of $\mathcal{H}(4)$ and of $\mathcal{H}(3,1)$, respectively. The kernels of both the monodromies contain a non-abelian free group of rank 2. 
\end{alphateo}

The orbifold fundamental groups involved in the statement of  Theorem \ref{ker} are closely related to Artin groups. Looijenga-Mondello showed that the groups $\pi_1^{orb}(\mathcal{H}^{nh}(4))$ and $\pi_1^{orb}(\mathcal{H}(3,1))$ are infinite-cyclic central extensions of the inner automorphism groups of some Artin groups \cite{Looijenga2014}. It turns out that Theorem \ref{ker} is an example of a more general phenomenon related to \textit{geometric homomorphisms} from Artin groups to mapping class groups.

\textbf{Geometric homomorphisms.} If $\Gamma$ is a finite, connected and undirected simple graph with $\mathcal{V}(\Gamma)$ as its set of vertices, an {Artin group} is a group that admits a presentation of the following form
\begin{align}
    A(\Gamma):=\Bigg\langle a_1,\dots,a_n\in\mathcal{V}(\Gamma)\Bigg\mid
    \text{ }
    \begin{matrix}
        a_ia_ja_i=a_ja_ia_j & \text{if }a_i\text{ and }a_j\text{ are adjacent}\\
        a_ia_j=a_ja_i & \text{otherwise}
    \end{matrix}
    \text{ }
    \Bigg\rangle. \label{presArt} 
\end{align}

Roughly speaking, a geometric homomorphism $A(\Gamma)\rightarrow\operatorname{Mod}_{g,n}$ arises as the correspondence between the vertices of the defining graph $\Gamma$ and a family of simple closed curves on the surface $\Sigma_{g,n}$.  The standard Artin generators in the presentation (\ref{presArt}) map to Dehn twists about curves that respect the intersection pattern given by the defining graph; see Figure \ref{perron1}.

\begin{figure}[h]
    \centering
    \includegraphics[scale=0.2]{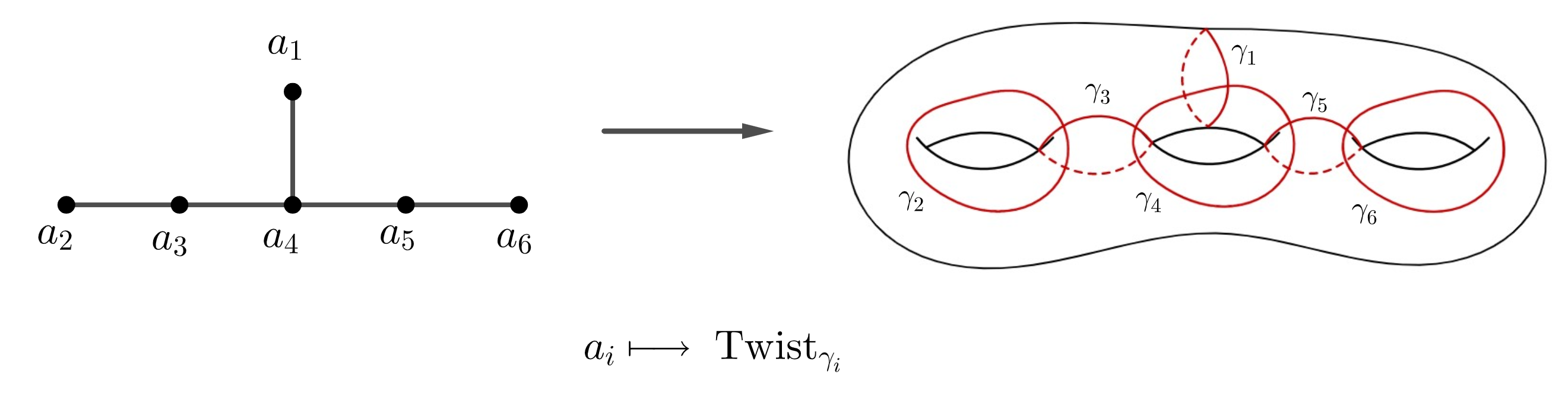}
    \caption{The map sending every standard generator $a_i$ to the Dehn twist $T_{\gamma_i}$ determines a geometric homomorphism of $A(E_6)$ to $\operatorname{Mod}_{3}$.}
    \label{perron1}
\end{figure}

Possibly, there might exist relations between Dehn twists that do not hold for standard generators of Artin groups. However, there is no known algorithm that can solve the word problem for a generic Artin group (for further details, see, for example \cite[Conjecture 5.2]{McCammond2017}), and this is the main obstruction to characterize kernels of geometric homomorphisms. 

Wajnryb proved that if the graph $\Gamma$ contains $E_6$ as a subgraph, any geometric homomorphism cannot be an injection \cite{Wajnryb1999}. In particular, Wajnryb found an element $w$ given explicitly in terms of the standard generators in the presentation (\ref{presArt}) and adopted the following strategy: as every inclusion of graphs induces a monomorphism of the respective Artin groups \cite{lek}, it is enough to find a non-trivial element $w$ in $A(E_6)$ which can be written in $\operatorname{Mod}_{3,1}$ as a braid relator of Dehn twists. The group $A(E_6)$ is a spherical-type Artin group, a class of groups for which the word problem has been solved by means of their \textit{Garside structure}. Our next result builds on Wajnryb's work and Theorem \ref{ker} can be thought of as a corollary of the following theorem.

\begin{alphateo}
\label{thmb}
    Let $\Gamma$ be any finite and undirect simple graph with $E_6$ as a subgraph. Any geometric homomorphism of $A(\Gamma)$ in $\operatorname{Mod}_{g,b}$ has a large kernel that contains a non-abelian free group $F_2$ of rank $2$. In particular, there is some $g\in A(\Gamma)$ such that $F_2$ is generated by the Wajnryb element $w$ and its conjugate $g^{-1}wg$.
\end{alphateo}

Theorem \ref{thmb} follows from the acylindrical hyperbolicity of spherical-type Artin groups modulo their center. Here, the Ping-Pong strategy can be adopted to detect non-abelian free groups.

\textbf{Acylindrical hyperbolicity.} Let $A(\Gamma)_\Delta$ denote the spherical-type Artin group $A(\Gamma)$ quotient by its center. Calvez-Wiest proved that the group $A(\Gamma)_\Delta$ acts \textit{acylindrically} on a $\delta$-hyperbolic graph, which is known in the literature as the \textit{additional length graph} $\textit{C}_{AL}(\Gamma)$ \cite[Theorem 1.3]{CalvezWiestAcyArt2017}. 

Calvez-Wiest found a group element $\kappa\in A(\Gamma)$ representing a loxodromic isometry of $\textit{C}_{AL}(\Gamma)$ that acts weakly properly discontinuously. By Osin's criterion \cite[Theorem 1.2]{Osin2016} the existence of the \textit{Calvez-Wiest} element $\kappa$ is enough to conclude the acylindrical hyperbolicity of $A(\Gamma)_\Delta$.

We prove that the infinite order Wajnryb element acts elliptically on $\textit{C}_{AL}(\Gamma)$ and a classical result shows that $w\in A(\Gamma)_\Delta$ cannot fix $\kappa\in A(\Gamma)_\Delta$ in the Gromov boundary of the additional length graph \cite[Lemma 25]{AntolinCumplidoParabolic21}. The following is due to Abbott-Dahmani and is the key ingredient we need to prove Theorem \ref{thmb}.

\begin{proposition*}[{\cite[Proposition 2.1]{AbbottDahmaniPnaive2019}}]
\label{quasi}
Let $G$ be a group acting acylindrically hyperbolic on a geodesic $\delta$-hyperbolic space $X$. Suppose $\sigma\in G$ is elliptic and $\gamma\in G$ is loxodromic. If 
\begin{enumerate}
    \item the set $\operatorname{A}_{10\delta}(\gamma)=\{x\in X\mid d(x,\gamma x)\leq \inf_{y\in X}d(y,\gamma y)+10\delta\}$ is not preserved by any non-trivial power of $\sigma$ and
    \item the diameter of $\operatorname{Fix}_{50\delta}(\sigma)=\{x\in X\mid d(x,\sigma^nx)\leq 50\delta\mbox{ for all } n\in\ZZ\}$ is finite,
\end{enumerate}
then there is some $n\in\ZZ$ such that the group generated by $\sigma$ and $\gamma^n$ is a non-abelian free group of rank $2$.
\end{proposition*}

We conclude that there exists a positive integer $n$ such that the group generated by $\kappa^{-n}w\kappa^n$ and $w$ is a non-abelian free group of rank $2$. 

\textbf{Projective strata.} We now explain how Artin groups arise in the context of the non-hyperelliptic components of the strata mentioned in Theorem \ref{ker}. 

The multiplicative group $\CC^*$ acts on the cotangent bundle of each Riemann surface $X$ by multiplication. The action preserves the multiplicity at the cone points of each holomorphic $1$-form and is well-defined of each connected component $\mathcal{C}$ of a stratum $\mathcal{H}(k_1,\dots,k_n)$. The resulting quotient is denoted by $\mathbb{P}\mathcal{C}$ and is known as a \textit{projective stratum} of translation surfaces. 

Looijenga-Mondello proved that the orbifold fundamental groups of $\mathbb{P}\mathcal{H}^{nh}(4)$ and $\mathbb{P}\mathcal{H}(3,1)$ are the inner automorphism groups of the $E_6$-type and $E_7$-type spherical Artin groups, respectively \cite{Looijenga2014}. A result of Pinkham implies that the monodromy map of $\mathbb{P}\mathcal{H}^{nh}(4)$ is geometric \cite{Pinkham}, meaning that standard Artin generators representing classes of elements in $\operatorname{Inn}(A(E_6))$ are mapped to some Dehn twists. We prove that the same holds for the monodromy of $\mathbb{P}\mathcal{H}(3,1)$. 

\begin{alphateo}
\label{main}
The topological monodromy $\rho_{\mathbb{P}\mathcal{H}(3,1)}:\pi_1^{orb}(\mathbb{P}\mathcal{H}(3,1))\rightarrow\operatorname{Mod}_{3,2}$ maps the classes of the standard generators to Dehn twists.
\end{alphateo}

Theorem \ref{main} and Pinkham's result are then enough to conclude that the kernels of the monodromies associated with the strata $\mathcal{H}^{nh}(4)$ and $\mathcal{H}(3,1)$ both contain a copy of a non-abelian free group $F_2$ of rank $2$.

\textbf{Structure of the paper} The paper is organized into five sections. In Section $1$ we describe the additional length graph associated with the Garside structure of a spherical-type Artin group $A(\Gamma)$, which serves as $\delta$-hyperbolic metric space for the acylindrical action of $A(\Gamma)$. In Section $2$ we define geometric homomorphisms and describe the Wajnryb element, while in Section $3$ we prove Theorem \ref{thmb}. In Sections $4$ and $5$ we draw the consequences that the existence of a non-abelian free group of rank $2$ implies for the topological monodromy of strata of abelian differentials. In particular, in Section $5$ we prove Theorem \ref{ker} and \ref{main}.

\textbf{Acknowledgments.} I am grateful for the patient guidance and support provided by my supervisors Tara Brendle and Vaibhav Gadre throughout working on this paper. I would also like to thank Aaron Calderon for the many helpful conversations, as well as Matt Bainbridge and Bradley Zykoski for patiently answering my questions. I am grateful to Aaron Calderon and Nick Salter for pointing me out the relation between versal deformation spaces of plane curve singularities and the flat geometry in genus $3$, which is the content of their joint work with Pablo Portilla Cuadrado. I would also thank Dawei Chen and the anonymous referee for their comments on the initial draft of this work. Finally, I would like to express my gratitude for the comments and suggestions provided by Aitor Azemar, Philipp Bader, Jim Belk, Rachael Boyd, Tudur Lewis, Miguel Orbegozo Rodriguez, Francesco Pagliuca and Franco Rota. 

\section{Spherical-type Artin groups}

Artin groups are finitely presented groups where the generators and the relations are given by a finite graph, as in (\ref{presArt}). For example, a braid group $\mathcal{B}_n$ is an Artin group with defining graph $A_{n-1}$, as in Figure \ref{spherical}. 

The quotient of a braid group $\mathcal{B}_n$  by the subgroup normally generated by the squares of the standard generators is the symmetric group of size $n$ \cite[Section 9.3]{farb2011primer} . Similarly, any Artin group $A(\Gamma)$ comes with a Coxeter group $W(\Gamma)$ given by the additional relations $a_i^2=1$ for every standard generator $a_i$. An Artin group $A(\Gamma)$ is of \textit{spherical-type} if $W(\Gamma)$ is finite. 
\begin{figure}[h]
    \centering
    \includegraphics[scale=0.25]{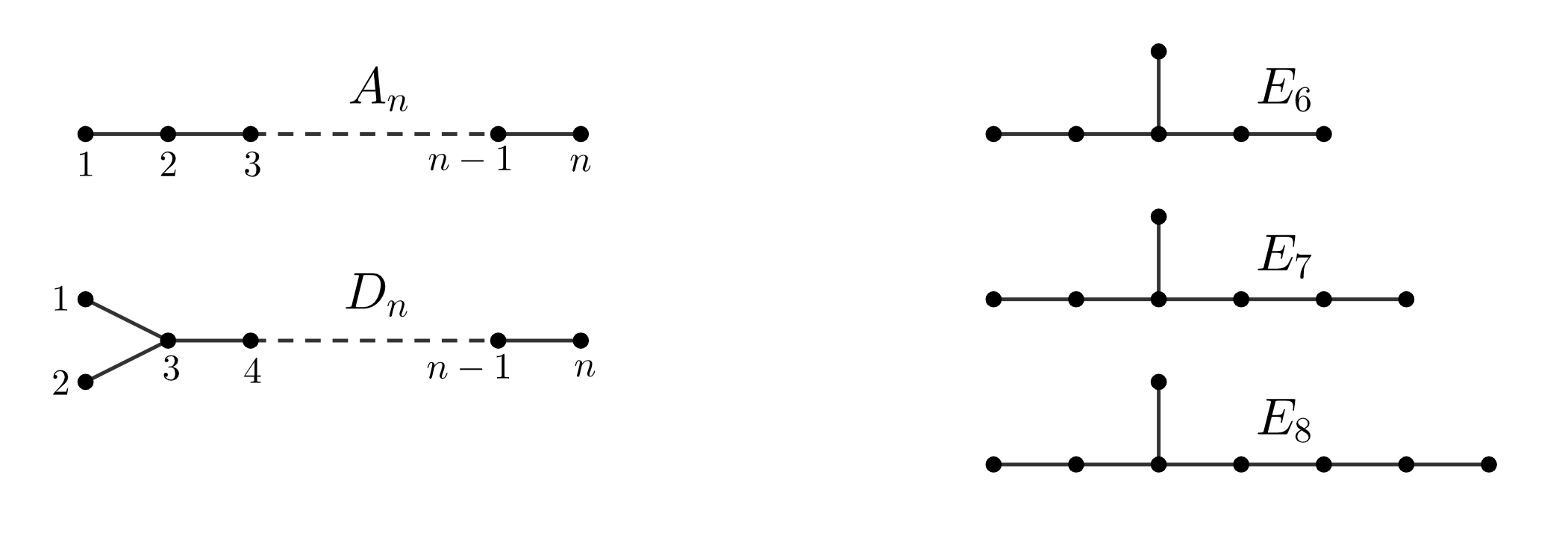}
    \caption{The spherical Artin groups. In particular, the $E_6$-type and $E_7$-type groups that describe the orbifold fundamental group of the projective components $\mathbb{P}\mathcal{H}(3,1)$ and $\mathbb{P}\mathcal{H}^{nh}(4)$ are spherical.}
    \label{spherical}
\end{figure}

Spherical-type Artin groups are well understood and, for example, admit a solution to the word problem. The algorithm is given by the Garside structure; see \cite{BrieskornArtin1972} for more details.

\textbf{Garside groups.} Let $G$ be a finitely generated group and $G^+$ the submonoid generated by the same finite generating set of $G$. Furthermore, suppose that $G^+$ trivially intersects $(G^+)^{-1}$. The prefix order on $G$ is the partial order $(G,\preceq)$ where  $a\preceq b$ if and only if $a^{-1}b\in G^+$.

If $\mathcal{A}$ is the set of elements in $G^+$ that cannot be written as a product of other non-trivial elements of $G^+$, the monoid $G^+$ is \textit{Noetherian} if for every $x\in G^+$ we have that $\sup\{n\in\NN\mid x=a_1\dots a_n, a_i\in\mathcal{A}\}$ is finite. Then, the group $G$ is \textit{Garside} if its monoid $G^+$ is Noetherian, the prefix order admits greater common divisors and lower common multiples and there exists  an element $\Delta\in G^+$ such that:
\begin{itemize}
    \item the conjugation action by $\Delta$ fixes the monoid $G^+$;
    \item the set $\{s\in G\mid 1\preceq s\preceq \Delta\}$ of simple elements is finite and generates $G$.
\end{itemize}

All the above assumptions guarantee that every $x\in G$ in a Garside group can be uniquely written in its \textit{(left) normal form} as $x=\Delta^ks_1\dots s_n,$ where each $s_i$ is a simple element that is not $\Delta$, and for every pair $\{s_i,s_{i+1}\}$ of adjacent simple elements the greater common divisor between $s_is_{i+1}$ and $\Delta$ is exactly $s_i$. The integer $k$ is denoted by $\inf(x)$, while $n$ is denoted by $\sup(x)$. 

Similarly, the suffix order on a Garside group $G$ provides the group with right normal forms, where elements in $G$ can be uniquely written in the form $t_1\dots t_n\Delta^k$ with $t_i$ simple elements different from $\Delta$; see \cite[Section 2]{AntolinCumplidoParabolic21} for more details. In what follows, we will only adopt the prefix order.

We say that $x\in G$ absorbs $y\in G$ if either $\sup(y)=0$ or $\inf(y)=0$ and both the equalities $\sup(xy)=\sup(x)$ and $\inf(xy)=\inf(x)$ hold. In this case, $y$ is absorbed by $x$ and we say that the group element $x$ is \textit{absorbable}. 

Spherical-type Artin groups admit a Garside structure through the submonoid $A(\Gamma)^+$ generated by the same standard generators of $A(\Gamma)$. Simple elements of $A(\Gamma)^+$ are words with free-square subwords. Absorbable elements are not classified but, for example, if $A(\Gamma)$ is the Braid group $\mathcal{B}_4$, any $n$-th powers of a generator absorb the $n$-th power of a non-adjacent generator. More precisely, we can write $$\sigma_1^n\sigma_3^n=(\sigma_1\sigma_3)^n,$$ and observe that $\sigma_1^n$ absorb $\sigma_3^n$.

The \textit{Garside element} $\Delta$ is the least common multiple of all the standard generators and any spherical type Artin group has infinite cyclic center generated by a power of $\Delta$ \cite[Théorème 7.1]{BrieskornArtin1972}. In what follows, we denote by $A(\Gamma)_\Delta$ the $\Gamma$-type spherical Artin group modulo its center. 

\textbf{Acylindrical hyperbolicity.} Let $\mathcal{S}$ be the set of simple and absorbable elements of a spherical-type Artin group $A(\Gamma)$. The vertices of the \textit{additional length graph} $\mathcal{C}_{AL}(\Gamma)$ are the left cosets of the subgroup $\langle\Delta\rangle$ in $A(\Gamma)$. Two cosets $g_1\langle\Delta\rangle$ and $g_2\langle\Delta\rangle$ are adjacent if they differ by the left multiplication of some element in $\mathcal{S}\setminus\{\Delta\}$.  As usual, the graph comes equipped with the metric where each edge has length one.

Calvez-Wiest proved that $\mathcal{C}_{AL}(\Gamma)$ is a $\delta$-hyperbolic geodesic metric space \cite[Theorem 1]{CalvezWiestCurve2017} and $A(\Gamma)_\Delta$ is an \textit{acylindrically hyperbolic} group for its action on $\mathcal{C}_{AL}(\Gamma)$  \cite[Theorem 1.3]{CalvezWiestAcyArt2017}. More precisely, they proved the following theorem.

\begin{theorem}
\label{acyartin}
If $A(\Gamma)$ is a spherical-type Artin group, then the action of $A(\Gamma)_{\Delta}$ on $\mathcal{C}_{AL}(\Gamma)$ is cobounded and non-elementary. Moreover, for every $\varepsilon>0$ there is a positive real numbers $R(\varepsilon)$ such that for each $x,y\in \mathcal{C}_{AL}(\Gamma)$ with $d(x,y)>R(\varepsilon)$, the set 
$$\Gamma_{\varepsilon}(x,y)=\{g\in A(\Gamma)_\Delta\mid d(x,gx)<\varepsilon, d(y,gy)<\varepsilon\}$$
is finite. 
\end{theorem}

It follows from a more general theorem \cite[Theorem 1.1]{Osin2016} that every $g\in A(\Gamma)_\Delta$ is either a \textit{loxodromic} or an \textit{elliptic} isometry of $\mathcal{C}_{AL}(\Gamma)$. In other words, either
\begin{itemize}
    \item the map $n\mapsto g^n\cdot x$ is a quasi-isometry between $\ZZ$ and the orbit of some (equivalently any) point $x\in\mathcal{C}_{AL}(\Gamma)$, and then $g$ is \textit{loxodromic}, or
    \item the action by $g$ has bounded orbits, and $g$ is then\textit{ elliptic}.
\end{itemize}

Any loxodromic element $g\in A(\Gamma)_\Delta$ is always \textit{weakly properly discontinuous}: for every $\varepsilon>0$ and $x\in \mathcal{C}_{AL}(\Gamma)$ there exists some $n\in\ZZ$ such that the set
$$\{g\in G\mid d(x,gx)<\varepsilon, d(\kappa^nx,g\kappa^nx)<\varepsilon\}$$ is finite. It turns out that the existence of a loxodromic and weakly properly discontinuous group element is also a sufficient property to show that a non virtually-cyclic group is acylindrically hyperbolic \cite[Theorem 1.2]{Osin2016}. Calvez-Wiest proved Theorem \ref{acyartin} by showing that 
\begin{equation}
\label{kappa}
   \kappa=a_4a_1a_3a_2a_4a_5a_4a_1a_3a_2a_6a_5a_5a_6a_2a_3a_1a_4a_5a_4a_2a_3a_1a_4 
\end{equation}

projects to a loxodromic and weakly properly discontinuous isometry in $A(E_6)_\Delta$.

However, there is no known sufficient and necessary criterion to determine if a given isometry of an acylindrically hyperbolic group is loxodromic or elliptic. Nevertheless, Antolin-Cumplido gave a sufficient condition for an isometry of the additional length graph to have bounded orbits \cite[Theorem 2]{AntolinCumplidoParabolic21}.  We will describe this criterion below. 

A \textit{parabolic subgroup} $P$ of an Artin group $A(\Gamma)$ is the conjugate of a subgroup generated by some strict subset of the standard generators. If $P$ is not a direct product of non-trivial parabolic subgroups, we say that it is \textit{irreducible}. The complex of irreducible parabolic subgroups $\mathcal{P}(\Gamma)$ is defined to have irreducible parabolic subgroups as vertices. A set of vertices $\{P_1,\dots,P_n\}$ is an $n$-simplex if one of the following properties is satisfied for all $i\neq j$:
\begin{itemize}
    \item $P_i\subset P_j$ or $P_j\subset P_i$;
    \item $P_i\cap P_j=\{1\}$ and $[P_i,P_j]=1$.
\end{itemize}
The complex $\mathcal{P}(\Gamma)$ can detect elliptic isometries of $\mathcal{C}_{AL}(\Gamma)$.

\begin{theorem}
\label{ell}
Suppose $A(\Gamma)$ is an irreducible spherical-type Artin group with more than two standard generators. The elements preserving some simplex of $\mathcal{P}(\Gamma)$ act elliptically on $\mathcal{C}_{AL}(\Gamma)$.  In particular, the normalizers of parabolic subgroups act elliptically on $\mathcal{C}_{AL}(A)$.
\end{theorem}

In Section $3$, we are also going to use a technical lemma borrowed from Antolin-Cumplido paper \cite[Lemma 25]{AntolinCumplidoParabolic21}. This lemma gives the following estimate for $g\in A(E_6)$ infinite order element in the normalizer of a proper standard parabolic subgroup and $x\in\mathcal{C}_{AL}(E_6)$:
\begin{equation}
\label{tech}
    d(g\kappa^n x,\kappa^nx)\geq d(x,\kappa^n x)+K,
\end{equation}
for some constant $K>0$ and $\lvert n\rvert$ big enough.

\section{Geometric homomorphisms and the Wajnryb element}
Let $\Sigma_g^b$ a closed and oriented genus $g$ surface with $b$ boundary components. The mapping class group $\operatorname{Mod}_{g}^b$ is the group of isotopy classes of orientation-preserving diffeomorphism of $\Sigma_g^b$ that pointwise fix the boundary and where the isotopies are required to fix the components of $\partial \Sigma_g^b$ pointwise.

\begin{figure}[!tbp]
  \centering
  \begin{minipage}[b]{0.2\textwidth}
    \includegraphics[width=\textwidth]{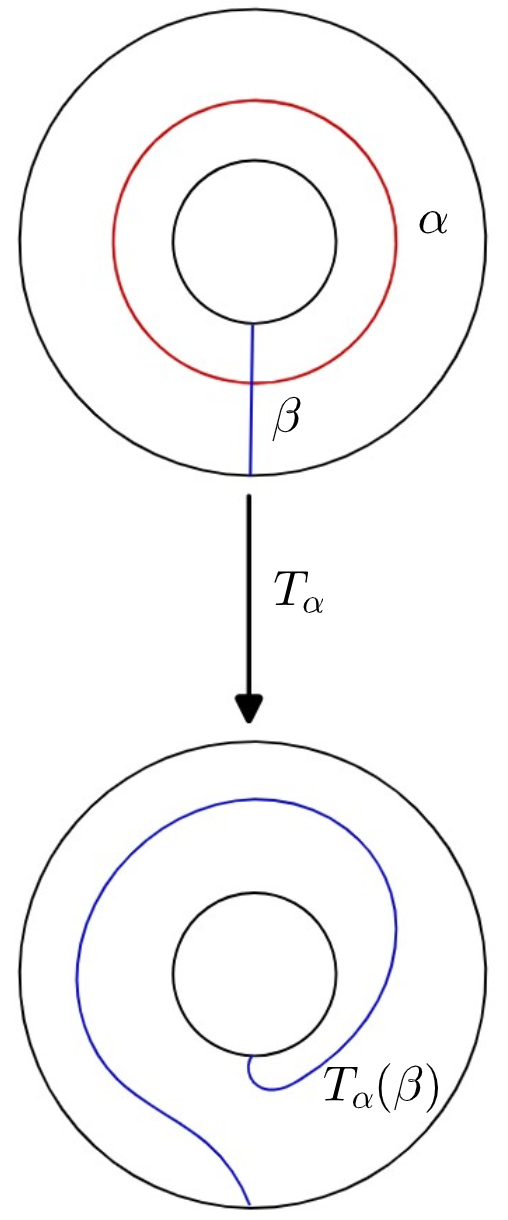}
  \end{minipage}
  \hfill
  \begin{minipage}[b]{0.5\textwidth}
    \includegraphics[width=\textwidth]{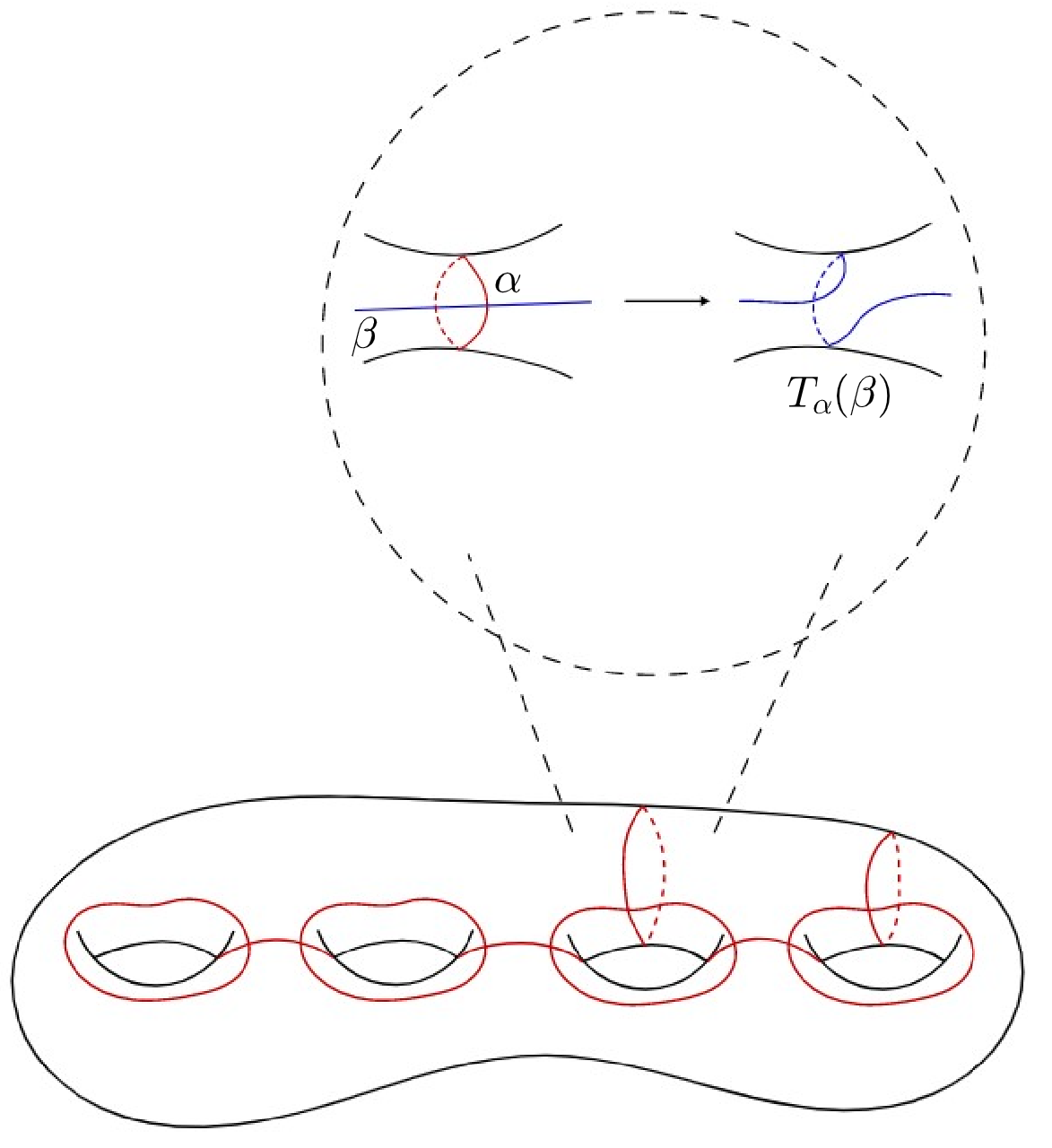}
  \end{minipage}
  \caption{The picture on the left represents the action of a Dehn twist about the curve $\alpha$ on the arc $\beta$ supported on an annulus. On the bottom right-hand side, the family of red simple closed curves represents the Humphries generating set for $\operatorname{Mod}_4$ consisting of Dehn twists about the $9$ curves shown. The top right-hand side of the picture shows the action of the Dehn twists about the red curve $\alpha$ on the blue curve $\beta$ that intersects $\alpha$ once.}
\end{figure}

\textit{Dehn twists} are diffeomorphisms of $\Sigma_g^b$ supported on the tubular neighborhood of some simple closed curve, as in Figure $3$. Two Dehn twists $T_{\gamma_1}$ and $T_{\gamma_2}$ commute if and only if $\gamma_1$ and $\gamma_2$ are disjoint, and satisfy the braid relation $T_{\gamma_1}T_{\gamma_2}T_{\gamma_1}=T_{\gamma_2}T_{\gamma_1}T_{\gamma_2}$ if and only if the geometric intersection number of $\gamma_1$ and $\gamma_2$ is exactly $1$. 

Suppose that $\Omega$ is a finite family of isotopy classes of non-essential simple closed curves on $\Sigma_g^b$. Moreover, suppose that the geometric intersection number of each pair of curves in $\Omega$ is at most $1$. The \textit{intersection graph} $\Lambda_\Omega$ of $\Omega$ is the graph with set of vertices $\Omega$ and edges for any pair of intersecting curves. Then, any pair of Dehn twists about curves in $\Omega$ either commute or satisfy the Braid relation. The group $$\operatorname{Mod}_g^b(\Omega)=\langle T_\gamma\mid\gamma\in\Omega\rangle$$ is then the quotient of the Artin group $A(\Lambda_\Omega)$ by some normal subgroup. We will call the quotient map
$$\varphi_{\Omega}:A(\Lambda_\Omega)\twoheadrightarrow\operatorname{Mod}_g^b(\Omega)$$ a \textit{geometric homomorphism}.

It is known that the $A_n$ and $D_n$ type Artin groups can be embedded into the mapping class group of some surfaces via a geometric homomorphism  \cite[Théorème 1]{Perron1996}. However, Wajnryb proved that for the $E_6$-type Artin group this is never the case \cite[Theorem 3]{Wajnryb1999}. 

\begin{theorem}
    Any geometric homomorphism of an Artin group $A(\Gamma)$ is not injective as long as $\Gamma$ contains $E_6$ as a subgraph.
\end{theorem}

Wajnryb found a non-trivial element $w\in A(E_6)$ that maps trivially in $\operatorname{Mod}_{3}^1$ via a geometric homomorphism $\varphi_\Omega$, where the set of curves $\Omega$ has intersection graph $E_6$. However, the result can be extended to Artin groups with defining graph $\Gamma$ containing $E_6$, as the embedding of $E_6$ in $\Gamma$ induces an inclusion of $A(E_6)$ in $A(\Gamma)$ and therefore an inclusion of mapping class groups. (see \cite{lek} and \cite[Theorem 3.18]{farb2011primer}).


With respect to Figure \ref{perron1}, the \textit{Wajnryb element} $w$ can be written as a word in the alphabet $\{a_1,b\}\subset A(E_6)$, where $$b=a_4a_5a_3a_4a_2a_6a_5a_3a_4$$ is contained in a parabolic subgroup isomorphic to $\mathcal{B}_6$. The element $b$ has the following braid representation.

\begin{figure}[h]
    \centering
    \includegraphics[scale=0.1]{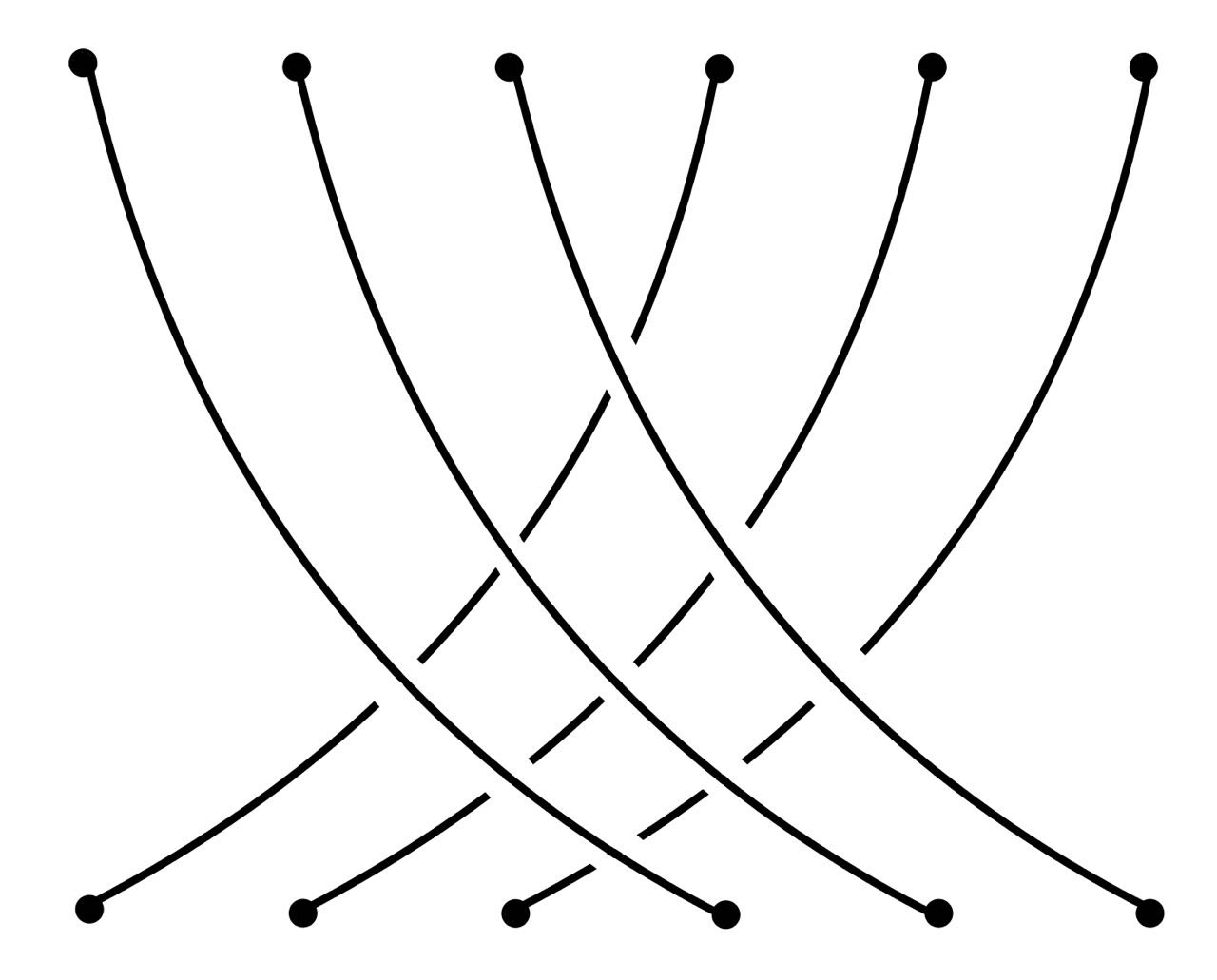}
    \caption{Braid representation of $b$.}
    \label{braidb}
\end{figure}

The image of $b$ via the geometric homomorphism $\varphi_\Omega$ is represented by a diffeomorphism which maps the simple closed curve $\gamma_1$ to a simple closed curve $\beta$ that intersects $\gamma_1$ once. Hence, the Dehn twists $T_{\gamma_1}$ and $T_\beta$ satisfy the braid relation and the Wajnryb element  
\begin{align*}
\label{wajnryb}
    w=a_1a_1^ba_1\cdot (a_1^{-1})^ba_1^{-1}(a_1^{-1})^b,
\end{align*}
acts trivially on $\Sigma_{3,1}$ as a mapping class. Indeed, its image via $\varphi_{\Omega}$ is precisely $T_{\gamma_1}T_\beta T_{\gamma_1}\cdot T_\beta^{-1}T_{\gamma_1}^{-1}T_\beta^{-1}=1$. Wajnryb proved that the group element $w$ is non-trivial applying the Garside algorithm. However, it is less clear why $w$ should geometrically describe a non-trivial group element in $A(E_6)$. 

\begin{figure}[h]
    \centering
    \includegraphics{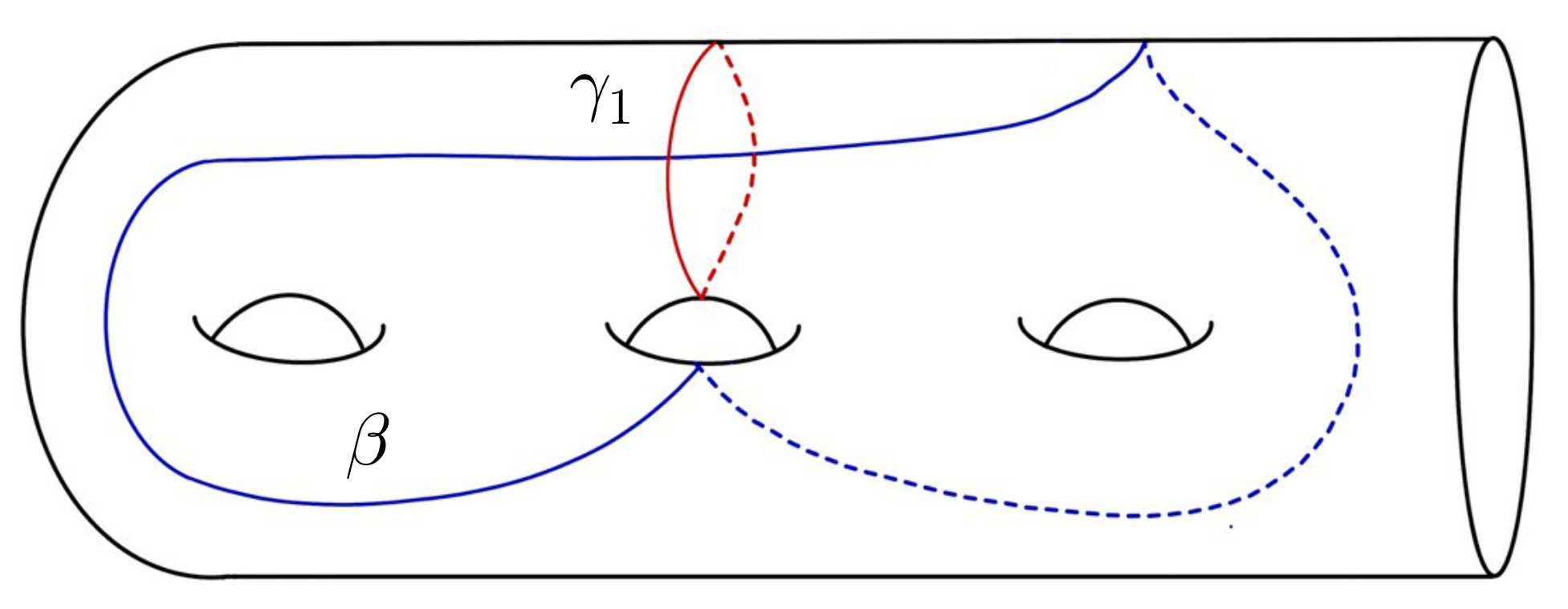}
    \caption{The curves $\beta$ and $\gamma_1$ in blue and red, respectively.}
\end{figure}

\section{Kernels of geometric homomorphisms}
In this section we construct a non-abelian free subgroup of rank 2 in the kernel of any geometric homomorphism of Artin group with defining graph containing $E_6$.

Recall that $\kappa\in A(\Gamma)_\Delta$ is the loxodromic isometry of $\mathcal{C}_{AL}(E_6)$ in (\ref{kappa}). In view of the Abbott-Dahmani result \cite[Proposition 2.1]{AbbottDahmaniPnaive2019}, we show that the Wajnryb element $w$ is an elliptic isometry of $\mathcal{C}(E_6)$, that none of its powers preserve the quasi-axis $A_{10\delta}(\kappa)$ and that $\operatorname{Fix}_{50\delta}(w)$ is a bounded set. It will follow that there is a power $n\in\ZZ$ such that the subgroup $\langle w,\kappa^n\rangle$ is a non-abelian free group of rank $2$. 

\begin{proof}[Proof of Theorem \ref{thmb}]
Let $\Omega$ be a collection of isotopy classes of non-essential simple closed curves on $\Sigma_g^b$ that pairwise intersect at most once, and suppose that its intersection graph $\Lambda_\Omega$ contains $E_6$. 

The hypotheses of  the Abbott-Dahmani result \cite[Proposition 2.1]{AbbottDahmaniPnaive2019} are satisfied for $w$ and $\kappa$ by Lemma \ref{l1},  Lemma \ref{l2} and  Lemma \ref{l3} below. However, the loxodromic isometry $\kappa$ is not in the kernel of $\varphi_\Omega$. Nevertheless, if we denote by $w^{\kappa^n}$ the conjugate $\kappa^{-n}w\kappa^n$, the group $\langle w,w^{\kappa^n}\rangle$ is contained in $\ker\varphi_\Omega$ and it is also isomorphic to $F_2$, as any combination of letters in $\{w,w^{\kappa^n}\}$ that represents a trivial word is also a combination of letters in $\{w,\kappa^n\}$.

\end{proof}

\begin{lemma}
\label{l1}
The projection of $w$ in $A(E_6)_\Delta$ is an elliptic isometry of the additional length graph $\mathcal{C}_{AL}(E_6)$.
\end{lemma}
\begin{proof}
We would like to apply the Antolin-Cumplido criterion from Theorem \ref{ell}. It is enough to show that the subgroup $\langle a_1,b \rangle$ normalizes the parabolic subgroup $\langle a_2,a_5\rangle$. The action of $b$ by conjugation on $A(E_6)$ permutes $a_2$ and $a_5$ (see Figure \ref{bnorm}). Since $a_1$ is in the centralizer of both $a_2$ and $a_5$, we can conclude that the group generated by $a_1$ and $b$ preserves the $2$-simplex $\{\langle a_2\rangle\ ,\langle a_5\rangle\}$ of the complex $\mathcal{P}(E_6)$.

\end{proof}

\begin{figure}[h]
    \centering
    \includegraphics[scale=0.2]{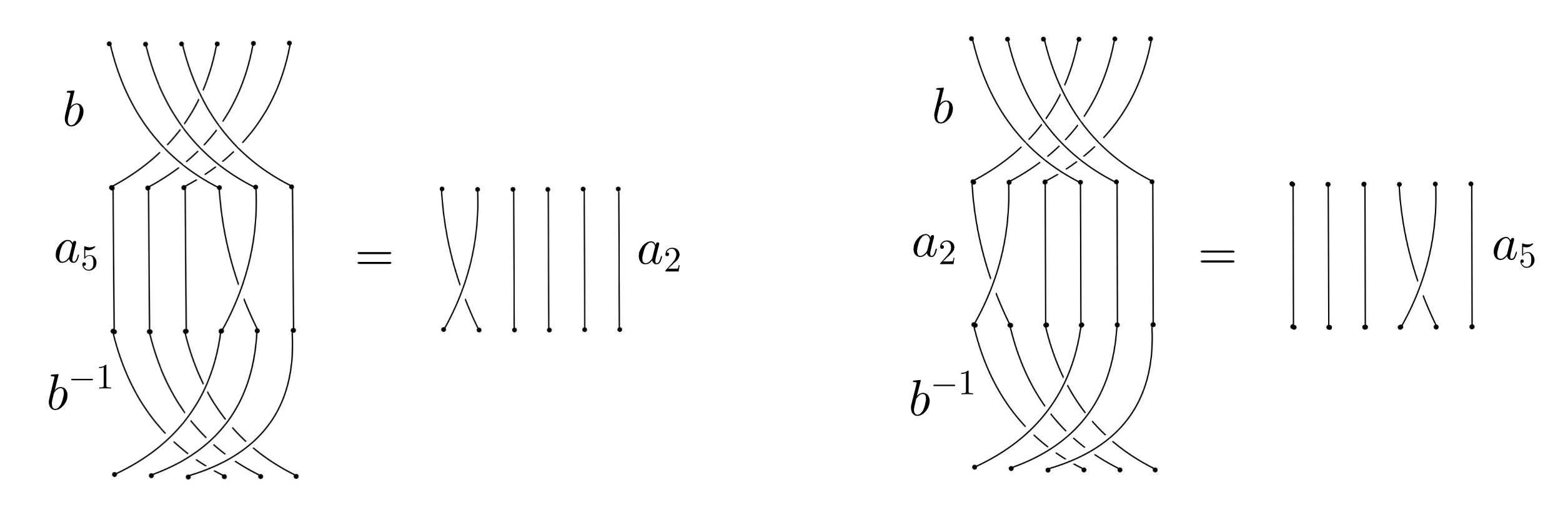}
    \caption{Braid representation of the conjugacy action of $b$ on $\langle a_2,a_5\rangle$.}
    \label{bnorm}
\end{figure}
\begin{lemma}
\label{l2}
    No non-trivial power of the Wajnryb element $w\in A(E_6)$ preserves the $10\delta$-quasi fixed axis $A_{10\delta}(\kappa)$ of $\kappa$.
\end{lemma}
\begin{proof}
    Let $x\in\operatorname{A}_{10\delta}(\kappa)$ be a vertex of $\mathcal{C}_{AL}(E_6)$. If we suppose that $w$, or any of its non trivial power, preserves $\operatorname{A}_{10\delta}(\kappa)$ we would have that
\begin{align*}
    d(w\kappa^nx,\kappa^nx)&\leq d(w\kappa^nx,\kappa^nwx)+d(\kappa^nw x,\kappa^nwx)& \text{(triangular inequality)}\\
    &=d(w^{\kappa^n} x, w x)+d(w x, x)&(\kappa\text{ is an isometry})\\
    &\leq  d(w^{\kappa^n} x, x)+2d(w x, x)&\text{(triangular inequality)}\\
    &\leq \inf_{y\in\mathcal{C}_{AL}(E_6)}d(y,\kappa y)+10\delta+2d(wx, x),&\text{(definition of }A_{10\delta}(\kappa))
\end{align*}
for any $n\in\ZZ$, where the last inequality follows from the fact that also $\kappa$ preserves $\operatorname{A}_{10\delta}(\kappa)$. However, the inequality (\ref{tech}) implies that it cannot happen, as $\kappa$ is loxodromic.

\end{proof}

Every spherical-type Artin group has a finite $K(\pi,1)$ space given by the complement of a hyperplane arrangement associated with the respective Coxeter group (see, for example, \cite{Deligne1972}). In particular, the Artin group $A(E_6)$ is torsion-free \cite[Proposition 2.45]{hatcher2002algebraic}. However,  the quotient $A(E_6)_\Delta$ has torsion elements but the Wajnryb element $w$ is not a periodic isometry of $\mathcal{C}_{AL}(E_6)$. 

In order to prove the following lemma, we recall that  standard generators $\{a_1,\dots,a_n\}$ of an Artin group $A(\Gamma)$ are related by length-preserving relations and the map \begin{align*}
    \operatorname{deg}:A(\Gamma)&\rightarrow\ZZ \\
    a_{i_1}^{n_1}\dots a_{i_k}^{n_k}&\mapsto \sum_{j=1}^k n_k
\end{align*}
is a homomorphism. More precisely, the commutator subgroup of $A(\Gamma)$ is exactly the kernel of the length homomorphism $\operatorname{deg}:A(\Gamma)\rightarrow\ZZ$ \cite[Proposition 3.1]{Mulholland2002}. 

\begin{lemma}
\label{l3}
    The Wajnryb element $w$ is not torsion in $A(E_6)_\Delta$.
\end{lemma}
\begin{proof}
Suppose there is some $m\in\ZZ$ such that $w^m$ is central in $A(E_6)$ and can be written as $\Delta^{k}$ for some integer $k$. The degree $\operatorname{deg}(w)$ is zero and therefore we can write $w$ as a commutator. However, the Garside element of $A(E_6)$ is $\Delta=(a_1a_3a_5a_2a_4a_6)^6$ and has positive length. Hence, we have that$$0=\operatorname{deg}(w^m)=\operatorname{deg}(\Delta^{k})=k\cdot\operatorname{deg}(\Delta)$$ and $k$ is then forced to be equal to zero. Since $A(E_6)$ is torsion-free, the only possibility for the  $m^{th}$-power of $w$ to be trivial is that $m=0$. 

\end{proof}

The set $\operatorname{Fix}_{50\delta}(w)$ is then necessarily bounded.

\begin{lemma}
    Let $G$ be a group acting acylindrically on a $\delta$-hyperbolic space $X$. If $\operatorname{Fix}_{K}(g)$ is unbounded, then $g$ has finite order.   
\end{lemma}
\begin{proof}
    Let $x,y\in \operatorname{Fix}_{K}(g)$ be two points of $X$ such that $d(x,y)$ is greater than the constant $R(K)$ from the definition of acylindrical hyperbolicity of a group (Theorem \ref{acyartin}). Then, the set $\Gamma_K(x,y)$ is finite and contains any power of $g$. 
    
\end{proof}

\section{The topological monodromy for strata of translation surfaces}

In this section we define the topological monodromy $\rho_\mathcal{C}:\pi_1^{orb}(\mathcal{C})\rightarrow\operatorname{Mod}_{g,n}$ of the connected components $\mathcal{C}$ of strata of translation surfaces. The relation between the connected components $\mathcal{H}^{nh}(4)$ and $\mathcal{H}(3,1)$ and the Artin groups of type $E_6$ and $E_7$ is stated at the end of this section.

\subsection{Translation surfaces as polygons}

As mentioned in the introduction, a translation surface on $\Sigma_g$ is defined by a genus $g$ Riemann surface $X$ and a holomorphic non-zero global section $\omega$ of the cotangent bundle of $X$, called \textit{abelian differential}. Each $\omega$ has $2g-2$ zeros on $X$ counted with multiplicity \cite[Theorem 1.2]{Wright2015}.

By means of a developing map, an isomorphic class of translation structure on $\Sigma_g$ is equivalently an equivalence class of polygons on the complex plane. The sides of the polygon are identified in pairs via translations, such that the quotient space is $\Sigma_g$. Two such polygons define isomorphic translation structures if one can be obtained from the other by a \textit{scissor move}, as shown in Figure (\ref{scissor}). This operation is performed by cutting one of the two polygons along a straight segment joining two vertices and gluing back the two cut pieces along identified sides via a translation. 

In Section $6$ we are going to use this alternative definition of translation surface to describe a generating set for the orbifold fundamental group $\pi_1^{orb}(\mathbb{P}\mathcal{H}^{nh}(4))$.

\begin{figure}[h]
    \centering
    \includegraphics[scale=0.2]{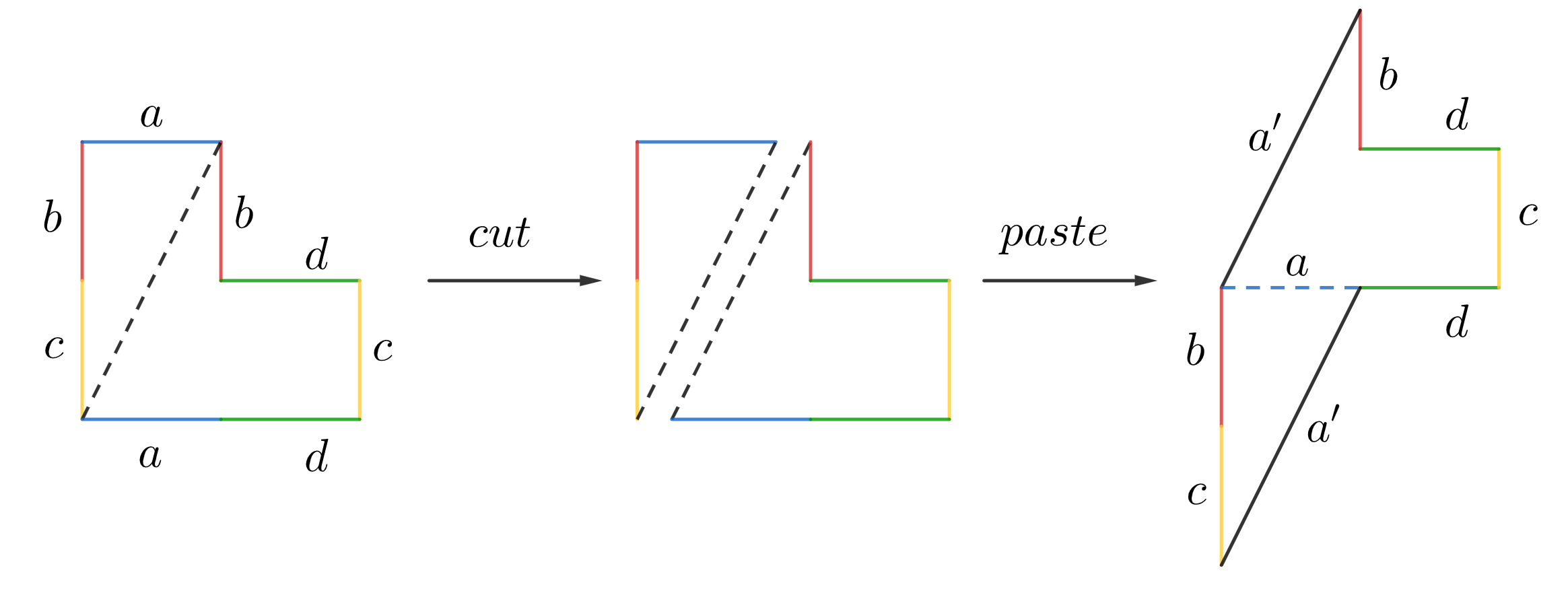}
    \caption{A scissor move for a translation surface of genus $2$. Sides with the same label are identified by a translation.}
    \label{scissor}
\end{figure}

\subsection{The strata of marked abelian differentials} 

Let $P$ be a finite set of points on $\Sigma_g$. If $(X,\omega)$ is a genus $g$ translation surface with $\mathcal{Z}(\omega)$ as set of zeros for $\omega$, a marking $f:(\Sigma_g,P)\rightarrow(X,\mathcal{Z}(\omega))$ is the isotopy class rel $P$ of a diffeomorphisms. The marked stratum $\mathcal{T}\mathcal{H}(\underline{k})$ is the set of triples $(X,f,\omega)$ where $(X,\omega)\in\mathcal{H}(\underline{k})$ and $f$ is a marking of $(X,\omega)$.

The topology of each stratum $\mathcal{H}(\underline{k})$ is inherited by its cover $\mathcal{T}\mathcal{H}(\underline{k})$.  Indeed, every \textit{marked stratum} $\mathcal{T}\mathcal{H}(\underline{k})$ is equipped with an atlas of charts in $\CC^{2g+n-1}$. Let $\tau$ be a triangulation of $\Sigma_g$ where the vertices are points in $P$. The set $U_\tau$ of triples $(X,f,\omega)\in\mathcal{T}\mathcal{H}(\underline{k})$ such that $f(\tau)$ is a triangulation of $(X,\omega)$ via saddle connections, namely geodesic arcs intersecting the zeros of $\omega$ only at the endpoints. If $\{\gamma_1,\dots,\gamma_{2g+n-1}\}$ is a fixed basis for the relative homology group $H_1(\Sigma_g,P,\ZZ)$, the charts are given by the maps
\begin{align*}
    U_\tau&\rightarrow H^1(\Sigma_g, P,\CC)\\
    (X,f,\omega)&\mapsto (\gamma_i\mapsto\int_{f_*\gamma_i}\omega)_{i=1}^{2g-n+1}
\end{align*}
\cite[Proposition 2.1]{BainbridgeSmillieWeissHorocycle2022}. 

The mapping class group $\operatorname{Mod}_g$ acts on the marked strata by precomposition on the markings and the resulting quotient space gives the quotient topology to the stratum $\mathcal{H}(\underline{k})$. However, the action of the mapping class group $\operatorname{Mod}_g$ is not free on the marked strata, but the point-stabilizers are finite groups \cite[Section 12.1]{farb2011primer}. In particular, the strata of translation surfaces are orbifolds.

In general, each $\mathcal{H}(\underline{k})$ is not connected and its number of connected components is at most $3$ \cite[Theorem 1]{Kontsevich2003}. The strata $\mathcal{H}(2g-2)$ and $\mathcal{H}(g-1,g-1)$ both
have a hyperelliptic connected component, which is isomorphic to quotients of configuration spaces of points on the Riemann sphere by the action of the group of some roots of unity \cite[Theorem 2.3]{CalderonConnected2020}. On the other hand, studying the topology of the non-hyperelliptic components proves to be more intricate.

\subsection{The topological monodromy map}

If $M$ is a connected manifold and $G$ acts smoothly and properly discontinuously on $M$, the quotient space $M/G$ is a (good) orbifold. The\textit{ orbifold fundamental group} $\pi_1^{orb}(M/G,p)$ based at $p\in M$ is the group of pairs $(\eta,g)$, where $g\in G$ and $\eta$ is a homotopy class of arcs with endpoints $p$ and $g\cdot p$. The group operation on $\pi_1^{orb}(M/G,p)$ is given by the composition law $(\eta_1,g_1)(\eta_1,g_1)=(\eta_1*(g_1\cdot\eta_2),g_1g_2)$.

Let $(X,\omega)$ be a translation surface in some connected components $\mathcal{C}$ of a stratum $\mathcal{H}(\underline{k})$ and let us fix a marked translation surface $(X,f,\omega)\in\mathcal{T}\mathcal{H}(\underline{k})$. If $\mathcal{M}_{g,n}$ is the moduli space of genus $g$ Riemann surfaces with $n$ marked points, the forgetful map $\mathcal{C}\rightarrow\mathcal{M}_{g,n}$ induces a homomorphism between orbifold fundamental groups
$$\rho_\mathcal{C}:\pi_1^{orb}(\mathcal{C},(X,f,\omega))\rightarrow\pi_1^{orb}(\mathcal{M}_{g,n},X),$$

where $\pi_1^{orb}(\mathcal{M}_g,X)$ is just $\operatorname{Mod}_{g,n}$ \cite[Section 12.5.3]{farb2011primer}. Geometrically, the homomorphism $\rho_C$ keeps track of the change of marking performed along a loop $(\eta,g)$. The translation structure carried along $\eta$ coincides at the endpoints of the path in $\mathcal{T}\mathcal{H}(\underline{k})$ but the marking might change. Indeed, every path $\eta$ in $\mathcal{T}\mathcal{H}(\underline{k})$ with endpoints $(X,f,\omega)$ and $(X,f',\omega)$ is mapped to the mapping class $f^{-1}f'$ by $\rho_C$.

Calderon-Salter proved that in genus $g\geq 5$ and for every non-hyperelliptic component $\mathcal{C}$, the image of $\rho_\mathcal{C}$ in $\operatorname{Mod}_{g,n}$ is a \textit{framed mapping class group} \cite[Theorem A]{CalderonSalterFramed2022}. The framed mapping class groups are stabilizers of winding number functions defined on the marked translation surface $(X,f,\omega)$ that serves as a base point for the monodromy $\rho_\mathcal{C}$. Calderon-Salter result does not cover the case of the strata $\mathcal{H}^{nh}(4)$ and $\mathcal{H}(3,1)$. However, in this article the focus is on the kernels of the monodromy maps and not on the images. 

\subsection{Projective strata of abelian differentials}

Any non-zero complex number is a composition of a rotation and a homothety that acts on the strata of translation surfaces by rotating and dilating the defining polygons. The $\CC^*$-action on the strata $\mathcal{H}(\underline{k})$ is continuous and preserves $\mathcal{Z}(\omega)$ pointwise for every $(X,\omega)\in\mathcal{H}(\underline{k})$. A \textit{projective stratum} $\mathbb{P}\mathcal{H}(\underline{k})$ is the quotient of $\mathcal{H}(\underline{k})$ by the action of $\CC^*$. Equivalently, every projective stratum $\mathbb{P}\mathcal{H}(\underline{k})$ parameterizes pairs $(X,D)$ where $X$ is a smooth projective curve and $D$ is a canonical positive divisor with prescribed multiplicities given by representative abelian differentials \cite{Looijenga2014}.
 
In particular, for any connected component $\mathcal{C}$ of the stratum $\mathcal{H}(\underline{k})$, the quotient map $q:\mathcal{C}\rightarrow\mathbb{P}\mathcal{C}$ induces a homomorphism between orbifold fundamental groups $q_*:\pi_1^{orb}(\mathcal{C})\rightarrow\pi_1^{orb}(\mathbb{P}\mathcal{C})$. The topological monodromy map $\rho_{\mathbb{P}\mathcal{C}}:\pi_1^{orb}(\mathbb{P}\mathcal{C})\rightarrow\operatorname{Mod}_g^n$ can then be also defined for $\mathbb{P}\mathcal{C}$ as for $\mathcal{C}$. The monodromies $\rho_\mathcal{C}$ and $\rho_{\mathbb{P}\mathcal{C}}$ fit inside the commutative diagram below

\[
    \begin{tikzcd}[row sep=2em]
\pi_1^{orb}(\mathcal{C}) \arrow[rr, "q_*"] \arrow[dr, "\rho_C"']& & \pi_1^{orb}(\mathbb{P}\mathcal{C})\arrow[ld, "\rho_{{\mathbb{P}\mathcal{C}}}"]\\
 & \operatorname{Mod}_g^n.
\end{tikzcd}
    \]

For the stratum components in question, Looijenga-Mondello proved the following \cite{Looijenga2014}.

\begin{theorem}
\label{lomo}
    The orbifold fundamental groups of the projective connected components $\mathbb{P}\mathcal{H}^{nh}(4)$ and $\mathbb{P}\mathcal{H}(3,1)$ are isomorphic to $A(E_6)_\Delta$ and $A(E_7)_\Delta$, respectively.
\end{theorem}

\section{The geometric monodromy maps of $\mathbb{P}\mathcal{H}^{nh}(4)$ and $\mathbb{P}\mathcal{H}(3,1)$}

In this section we prove Theorem \ref{ker} and \ref{main}. In particular, we study the topological monodromy homomorphisms of the projective strata $\mathbb{P}\mathcal{H}^{nh}(4)$ and $\mathbb{P}\mathcal{H}(3,1)$. In Theorem \ref{ham} we recall the definition of some standard generators of $\pi_1^{orb}(\mathbb{P}\mathcal{H}^{nh}(4))$ that map via the homomorphisms $$\rho_{\mathbb{P}\mathcal{H}^{nh}(4)}:\pi_1^{orb}(\mathbb{P}\mathcal{H}^{nh}(4))\rightarrow\operatorname{Mod}_{3,1}$$ to Dehn twists. Discussions related to Theorem \ref{ham} are subject of an ongoing work by Calderon-Cuadrado-Salter; for instance, see \cite{Cuadrado2021}. The main ideas that underline Theorem \ref{ham} come from the theory of versal deformation spaces for plane curve singularities \cite{arnold}. For the sake of completeness, we are going to include a similar description of $\rho_{\mathbb{P}\mathcal{H}(3,1)}$ closely following the work of Calderon-Cuadrado-Salter. In particular, we prove that the monodromy $$\rho_{\mathbb{P}\mathcal{H}(3,1)}:\pi_1^{orb}(\mathbb{P}\mathcal{H}(3,1))\rightarrow\operatorname{Mod}_{3,2}$$ is geometric. However, the result of Calderon-Cuadrado-Salter gives explicit generators for $\pi_1^{orb}(\mathbb{P}\mathcal{H}^{nh}(4))$. This finite set of generators arises from the algebro-geometric theory of versal deformation spaces and can be described as a finite set of cylinder shears; see Figure \ref{shear} and \ref{e6trsu}. On the other hand, Theorem \ref{main} is enough to prove the existence of a non-abelian free group of rank $2$ in the kernel of $\rho_{\mathbb{P}\mathcal{H}(3,1)}$.

\begin{proof}[Proof of Theorem \ref{ker}]
    The copy of the non-abelian free group of rank $2$ we constructed in Theorem \ref{thmb} is in the kernel of any geometric homomorphism $A(E_6)\rightarrow\operatorname{Mod}_{3,1}$, but it is also a non-abelian free group in the kernel of the geometric map $A(E_6)_\Delta\rightarrow\operatorname{Mod}_{3,1}$, which is the same as $\rho_{\mathbb{P}\mathcal{H}^{nh}(4)}$. Similarly, the same copy of the non-abelian free group of rank $2$ can be defined in the kernel of the geometric homomorphism $A(E_7)_\Delta\rightarrow\operatorname{Mod}_{3,2}$, which is the same as $\rho_{\mathbb{P}\mathcal{H}(3,1)}$.
    
\end{proof}

\textbf{The monodromy of the stratum $\bm{\mathbb{P}\mathcal{H}^{nh}(4)}$.}  A \textit{cylinder} $\xi$ on a translation surface is an isometric embedding of a Euclidean cylinder whose boundary is a union of saddle connections. In particular, the interior of $\xi$ does not contain any singular point.

If the embedded cylinder $\xi$ is isometric to $(\mathbb{R}/a\mathbb{Z})\times[0,b]$ for some $a,b\in\RR^+$, the core curve of $\xi$ on the translation surface $(X,\omega)$ is the isotopic class of the simple closed curve which is the image of $(\mathbb{R}/a\mathbb{Z})\times\{t\}$ in $(X,\omega)$ for some $t\in(0,b)$. 

Suppose $\xi$ is a horizontal cylinder on a translation surface $(X,\omega)$. In particular, the cylinder $\xi$ can be represented as a rectangle $[0,b]\times[0,a]$ embedded in a defining polygon of $(X,\omega)$ with a pair of sides identified. Suppose that the ratio between its height $a$ and its weight $b$ is $R$. If $t\in[0,R]$, a \textit{cylinder shear} along $\xi$ is the result of the action by the matrix 
\begin{align*}
S_t=
    \begin{bmatrix}
    1& t\\
    0& 1
    \end{bmatrix}
\end{align*}
on the embedded parallelogram of the polygon representative. Analogously, by taking a suitable conjugate of $S_t$ one can define a cylinder shear along non-horizontal cylinders.

\begin{figure}[h]
    \centering
    \includegraphics[scale=0.25]{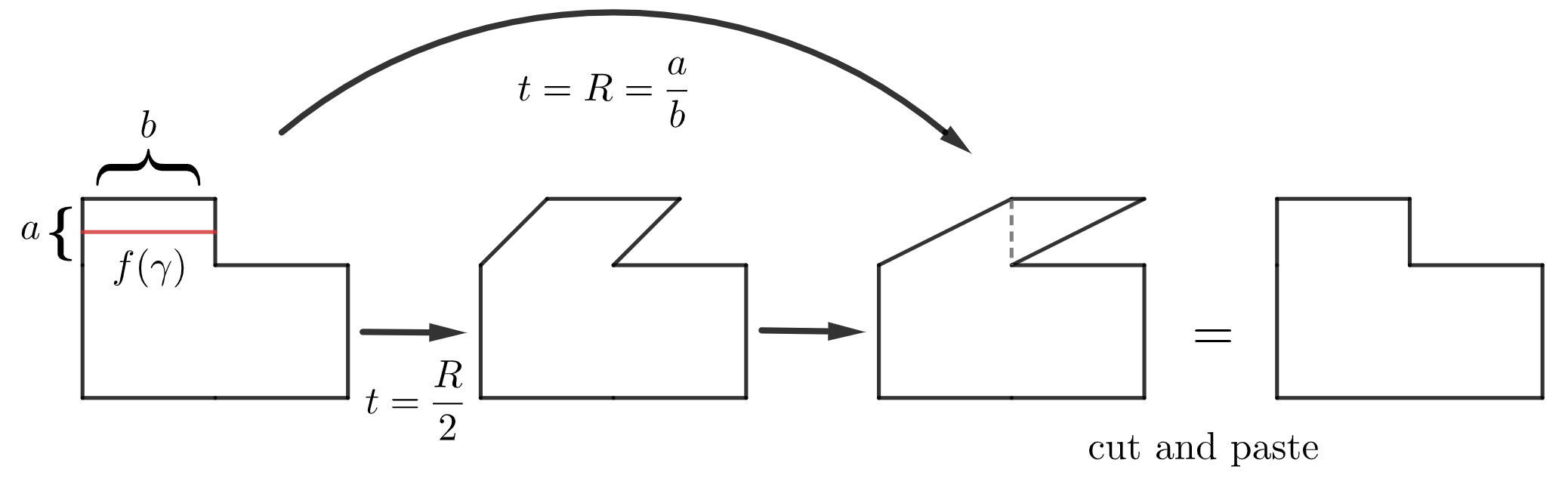}
    \caption{A full cylinder shear action on an L-shaped translation surface of genus 2, where opposite sides of the polygon are identified via a translation.}
    \label{shear}
\end{figure}

Let now $f:\Sigma_g\rightarrow X$ be a marking of $(X,\omega)$. The full shear $S_R$ acts on $(X,f,\omega)$ preserving the translation structure of $X$, as the resulting polygon differs from the initial one by a scissor move, as in Figure (\ref{shear}). However, the matrix $S_R$ changes the marking $f$ by a Dehn twist along the core curve of the cylinder $\xi$. Hence, a cylinder shear is an orbifold loop, and it is mapped via the topological monodromy map of the connected component containing $(X,\omega)$ to a Dehn twist.

The following result is known by experts. It appears as a consequence of Henry Pinkham's thesis \cite{Pinkham} and can also be found in \cite[Proposition 6.2]{ham}. 

\begin{theorem}
\label{ham}
    Let $\{\xi_1,\dots,\xi_6\}$ be a collection of embedded cylinders of a translation surface $(X,\omega)\in\mathcal{H}^{nh}(4)$ such that the family of the associated core curves have an $E_6$-type intersection graph. Then, there exists a map $\Theta:A(E_6)\rightarrow \pi_1^{orb}(\mathbb{P}\mathcal{H}^{hn}(4))$ that associates each standard generator of $A(E_6)$ to a full cylinder shear and can be extended to a well-defined surjective homomorphism with kernel the center of $A(E_6)$. 
\end{theorem}

\begin{figure}[h]
    \centering
    \includegraphics[scale=0.2]{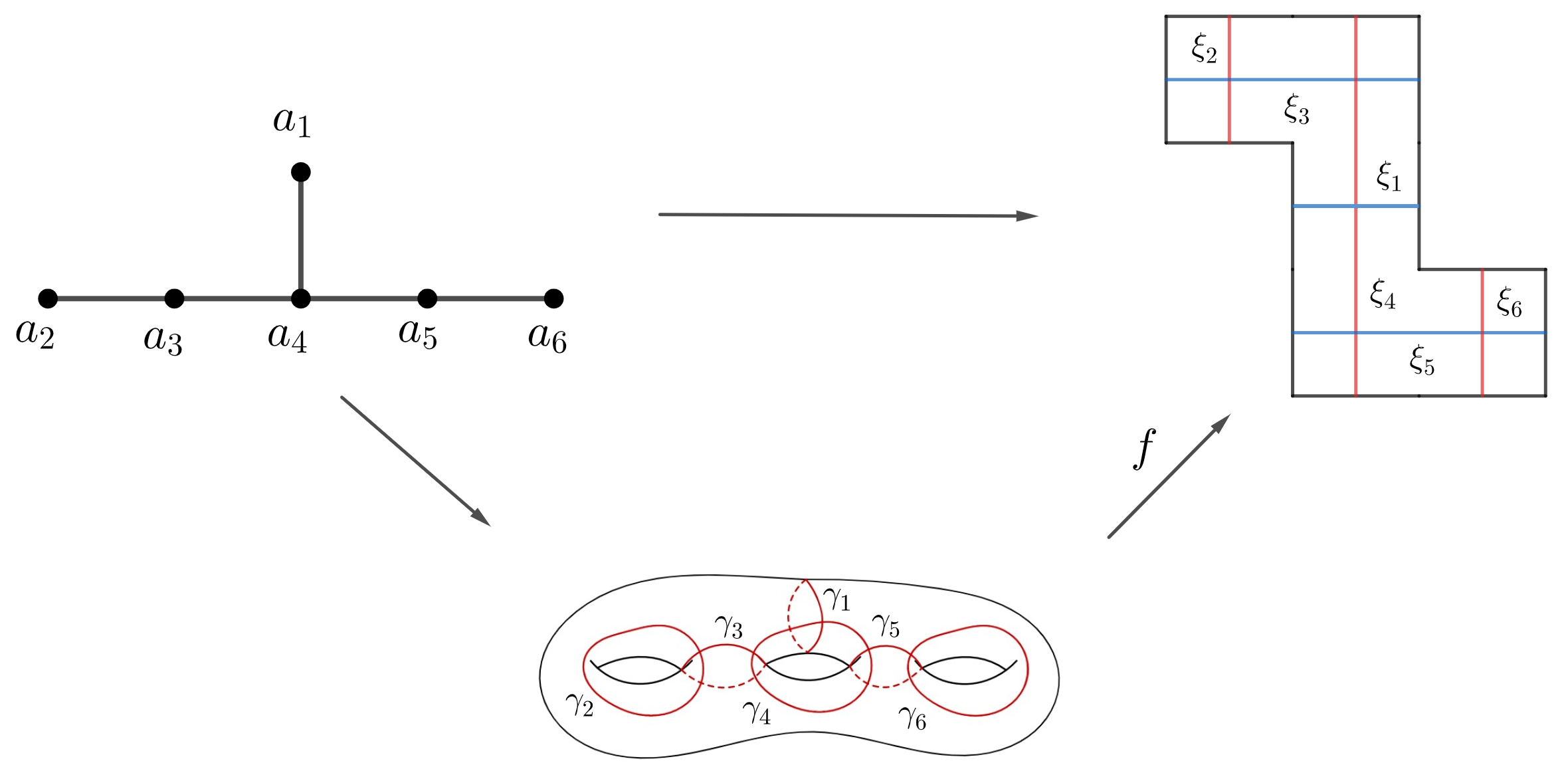}
    \caption{An $S$-shaped translation surface in $\mathcal{H}^{nh}(4)$. The red and blue segments represent the core curves of vertical and horizontal cylinders $\xi_i$, respectively. Their intersection graph is $E_6$.}
    \label{e6trsu}
\end{figure}

The homomorphism $\Theta$ is well-defined.  Every pair of adjacent standard generators in $A(E_6)$ is mapped to cylinder shears along embedded cylinders with core curves intersecting once; every pair of standard generators that commute is mapped to cylinder shears along disjoint flat cylinders.

Theorem \ref{ham} shows that $\pi_1^{orb}(\mathbb{P}\mathcal{H}^{nh}(4))$ is generated by a finite family of cylinder shears. However, the group $\pi_1^{orb}(\mathcal{H}^{nh}(4))$ contains orbifold loops that cannot be generated by cylinder shears only. These are loops that cyclically permute the prongs around the singularity. Calderon-Salter \cite[Corollary 7.6]{CalderonSalterFramed2022} showed that there exists an epimorphism $\pi_1^{orb}(\mathcal{H}^{nh}(4))\twoheadrightarrow\ZZ_2$ with kernel containing those orbifold loops that do not permute the prongs at the singularity. In particular, cylinder shears cannot cyclically permute any prong configuration. 

\textbf{The monodromy of the stratum $\bm{\mathbb{P}\mathcal{H}(3,1)}$.} 
Every genus $3$ non-hyperelliptic Riemann surface $X$ can be embedded in $\mathbb{C}\mathbb{P}^2$ as the vanishing locus of a smooth plane quartic \cite[Chapter VII, Proposition 2.5]{Miranda}. Such embedding is defined as the unique projective embedding of $X$ in $\mathbb{C}\mathbb{P}^2$ corresponding to the linear system of positive canonical divisors on $X$, up to linear change of coordinates. By abuse of notations, we identify every genus $3$ non-hyperelliptic Riemann surface $X$ with its image in $\mathbb{C}\mathbb{P}^2$. 

A \textit{flex} of a smooth quartic $X$ is a point $p\in X$ where the intersection multiplicity of $X$ with its tangent space is exactly $3$. A plane quartic $X$ with a flex point $p$ can always be reparametrized in such a way that $p$ is the point at infinity $[0:0:1]$ and its vanishing polynomial is of the form 
$$Q_s=x^3z+y^3x+s_1xyz^2+s_2xz^3+s_3y^4+s_4y^3z+s_5y^2z^2+s_6yz^3+s_7z^4\in\CC[x,y,z],$$
for some $s=(s_1,\dots,s_7)\in\CC^7$ \cite[Proposition 1]{Shioda1993}.

However, there are some strata $\mathcal{H}(k_1,\dots,k_n)$ where all the underlying Riemann surfaces are non-hyperelliptic. This is the case if all the odd numbers in the partition $(k_1,\dots,k_n)$ appear an odd number of times, since every positive canonical divisor on a hyperelliptic Riemann surface is the pullback of a divisor on the Riemann sphere $\mathbb{CP}^1$ \cite[Chapter IV, Proposition 5.3]{hartshorne}. In particular, the Riemann surfaces in the stratum $\mathcal{H}(3,1)$ are all non-hyperelliptic. 

\begin{proposition}

\label{flex}
Let $(X,\omega)$ be a translation surface in $\mathcal{H}(3,1)$. Then $X$ has a flex. In particular, the Riemann surfaces at each point in $\mathcal{H}(3,1)$ are vanishing loci $\mathbb{V}(Q_s)$ of quartics of the form $Q_s$, up to isomorphism.
    
\end{proposition}
\begin{proof}
    Since $X$ is a genus $3$ projective smooth curve, the positive canonical divisors associated with an abelian differential $(X,\omega)\in\mathcal{H}(3,1)$ coincide with divisors coming from lines $L_\omega$ in $\mathbb{CP}^2$ that intersects $X$ in two points. One of these points, say $p$, has multiplicity $3$; then $L_\omega$ is necessarily the tangent line to $X$ in $p$. In particular, $X$ has a flex at $p$ and is isomorphic to the vanishing locus of a quartic $Q_s$.
    
\end{proof}

On the other hand, every smooth vanishing locus $\mathbb{V}(Q_s)$ comes with an abelian differential in $\mathcal{H}(3,1)$ as follows.

The vanishing loci $\mathbb{V}(Q_s)$ are compact Riemann surfaces and the points at infinity can be removed to get a surface diffeomorphic to $\Sigma_{3,2}$. Equivalently, we can evaluate the homogeneous polynomial $Q_s$ at $z=1$ to get a polynomial $q_s\in\CC[x,y]$ and the respective affine vanishing locus $\mathbb{V}(q_s)$ in $\CC^2$.

Since every $\mathbb{V}(q_s)$ is the zero level set of a holomorphic function, the two complex derivatives $\partial_xq_s$ and $\partial_yq_s$ satisfy $\partial_xq_sdx+\partial_yq_sdy=0$. Moreover, the derivatives $\partial_xq_s$ and $\partial_yq_s$ cannot simultaneously vanish since $ \mathbb{V}(q_s)$ is smooth. Hence, the abelian differential 
\begin{align*}
    \omega_{s}(x_0,y_0)=
\begin{cases} 
       \frac{dx}{\partial_yq_s(x_0,y_0)} & \text{ if } \partial_yq_s(x_0,y_0)\neq 0 \\
      -\frac{dy}{\partial_xq_s(x_0,y_0)} & \text{ if } \partial_xq_s(x_0,y_0)\neq 0
   \end{cases}
\end{align*}
is well-defined and non-vanishing at every point $(x_0,y_0)\in\mathbb{V}(q_s)$. In particular, the volume form $\omega_s$ can be holomorphically extended to be zero on the two points at infinity $[1:0:0]$ and $[0:1:0]$, where $\omega_s$ vanishes with multiplicity $3$ and $1$, respectively. 

Looijenga observed that a line intersecting $\mathbb{V}(Q_s)$ with multiplicity $3$ is determined solely by the parameter $s$ \cite[Introduction]{Looijenga93}. In particular, up to a rescaling factor, the abelian differential $\omega_s$ is the unique holomorphic $1$-form on $\mathbb{V}(Q_s)$ such that the pair $(\mathbb{V}(Q_s),\omega_s)$ is a translation surface in $\mathcal{H}(3,1)$. 

In what follows we are going to denote by $\mathcal{M}_3^{\operatorname{flex}}$ the moduli space of non-hyperelliptic genus $3$ Riemann surfaces with $2$ marked points given by a flex $p\in X$ and the unique point of intersection between $T_pX$ and $X$ with multiplicity $1$ (Figure \ref{flexpic}).

\begin{figure}[h]
    \centering
    \includegraphics[scale=0.1]{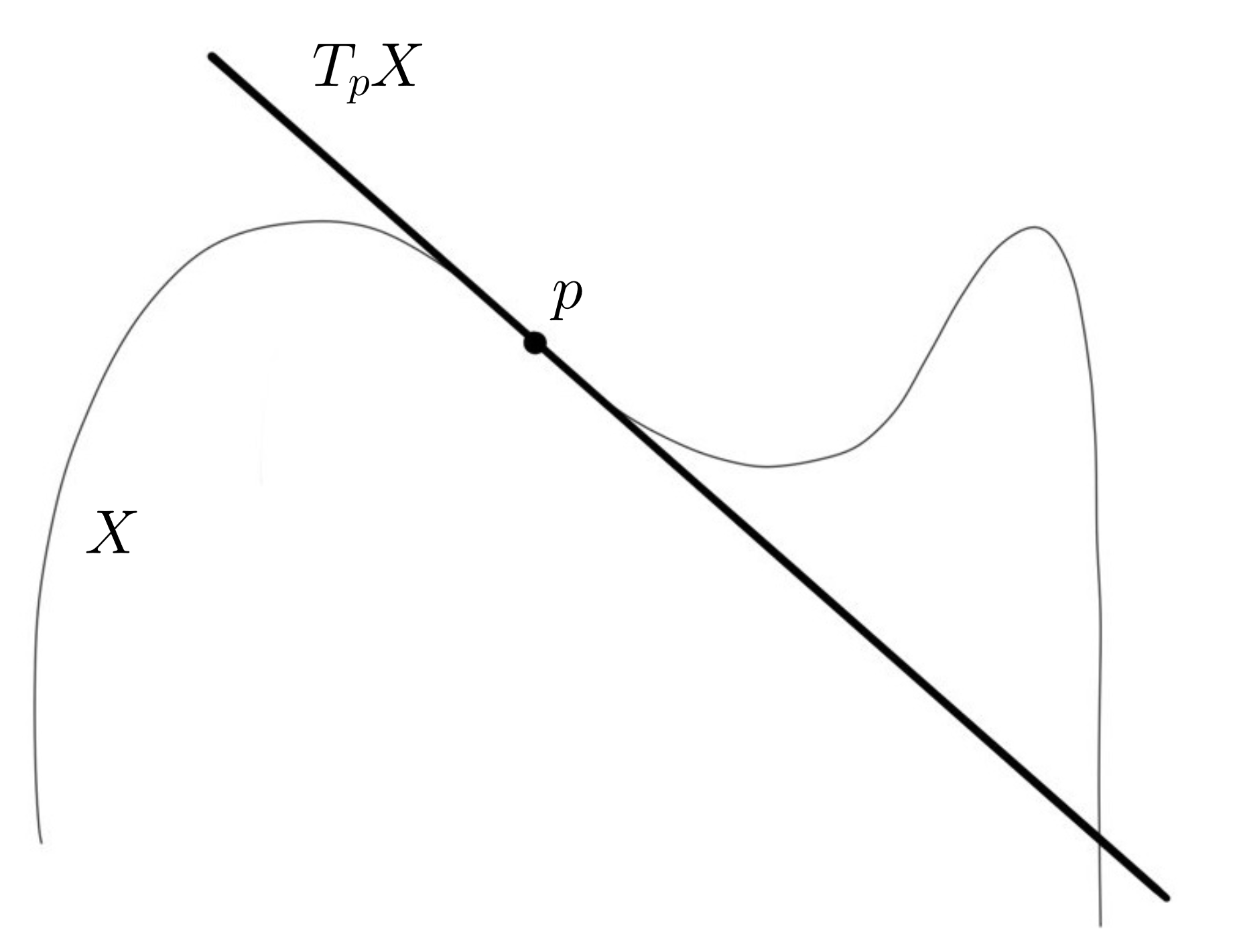}
    \caption{A quartic $X$ with a flex in $p$.}
    \label{flexpic}
\end{figure}

\begin{proposition}
\label{lo}
The forgetful map 
    \begin{align*}
        \mathbb{P}\mathcal{H}(3,1)&\rightarrow\mathcal{M}_3^{\operatorname{flex}}\\
        (X,[\omega])&\mapsto X
    \end{align*}
is an isomorphism on orbifolds. In particular, it induces an isomorphism $$\theta_1:\pi_1^{orb}(\mathbb{P}\mathcal{H}(3,1))\rightarrow\pi_1^{orb}(\mathcal{M}_3^{\operatorname{flex}})$$ that commutes with the monodromies $\rho^{\operatorname{flex}}:\pi_1^{orb}(\mathcal{M}_3^{\operatorname{flex}})\rightarrow\operatorname{Mod}_{3,2}$ and $\rho_{\mathbb{P}\mathcal{H}(3,1)}:\pi_1^{orb}(\mathbb{P}\mathcal{H}(3,1))\rightarrow\operatorname{Mod}_{3,2}$ of the respective moduli spaces.
    
\end{proposition}

\begin{proof}

    Let $\mathbb{P}\mathcal{H}^{\operatorname{Teich}}(3,1)$ be the Teichm\"{u}ller cover of the projective stratum $\mathbb{P}\mathcal{H}(3,1)$. If $\mathcal{T}_3^{\operatorname{flex}}$ is the Teichm\"{u}ller cover of $\mathcal{M}_3^{\operatorname{flex}}$, then the forgetful map
    \begin{align*}
        \mathbb{P}\mathcal{H}^{\operatorname{Teich}}(3,1)&\rightarrow\mathcal{T}_3^{\operatorname{flex}}\\
        (X,f,[\omega])&\mapsto(X,f)
    \end{align*}
    is a bijective quotient map and therefore a homeomorphism. Moreover, the monodromies $\rho^{\operatorname{flex}}$ and $\rho_{\mathbb{P}\mathcal{H}(3,1)}$ share the same image in $\operatorname{Mod}_{3,2}$ as every marking of a Riemann surface in $\mathbb{P}\mathcal{H}(3,1)$ appears as a marking of a Riemann surface in $\mathcal{M}_3^{\operatorname{flex}}$, and viceversa. Therefore, the forgetful map $\mathbb{P}\mathcal{H}^{\operatorname{Teich}}(3,1)\rightarrow\mathcal{T}_3^{\operatorname{flex}}$ induces an orbifold isomorphism and in particular an isomorphism between the respective orbifold fundamental groups.
    
\end{proof}

The collection of parameters $s\in\CC^7$ representing smooth quartics $Q_s$ is an Eilenberg-Maclane space for the spherical-type Artin group $A(E_7)$. In particular, the space $\{s\in\CC^7\mid  \mathbb{V}(Q_s)\text{ is smooth}\}$ has fundamental group isomorphic to $A(E_7)$ and can be homeomorphically realized as the complement of the complexified hyperplane arrangement $\cup_{i\in I}H_i$ of the root system $E_7$ modulo its reflection group $W(E_7)$ \cite[Proposition 9.3]{arnold}. We will briefly describe the homeomorphism between the space $\{s\in\CC^7\mid  \mathbb{V}(Q_s)\text{ is smooth}\}$ and the quotient of $\CC^7\setminus\cup_{i\in I}H_i$ by the group $W(E_7)$, as the construction is going to be used in Lemma \ref{sper}. 

The $\CC$-algebra of $W(E_7)$-invariant polynomials in $\CC[x_1,\dots,x_7]$ is generated by some homogeneous polynomials $q_1,\dots,q_7$ with degrees $d_i=\operatorname{deg}(q_i)$ uniquely determined by the finite group $W(E_7)$. The basis $\{q_1,\dots,q_7\}$ maps (in a neighborhood of zero) the quotient space $\CC^7/W(E_7)$ to $\CC^7$ by an isomorphism $\tau:\CC^7/W(E_7)\rightarrow\CC^7$ of complex manifolds. In particular, the image of the hyperplane arrangement $\cup_{i\in I}H_i$ modulo $W(E_7)$ is the hypersurface $\Pi=\{s\in\CC^7\mid  \mathbb{V}(Q_s)\text{ is singular}\}$ defined as the vanishing locus of a weighted homogeneous polynomial with weights given by the degrees $(d_1,\dots,d_7)$ of the homogeneous polynomials $\{q_1,\dots,q_7\}$; see, for example, \cite[Introduction]{discriminant} or \cite[Chapter 3]{arnoldbook}.

The complement $\mathbb{C}^7\setminus{\Pi}$ comes with a surface bundle with fibers diffeomorphic to $\Sigma_3^2$, as follows; for more details, see \cite[Section 2]{Cuadrado2021}. The intersection of the space $\{(p,s)\in\CC^2\times(\CC^7\setminus\Pi)\mid p\in\mathbb{V}(q_s)\}$ with a sufficiently small closed polydisk $\mathbb{D}^2\times\mathbb{D}^7$ in $\CC^2\times\CC^7$ is the total space of a fiber bundle with base space $\mathbb{C}^7\setminus{\Pi}$ and fibers diffeomorphic to $\Sigma_3^2$. It turns out that the monodromy $$\rho:\pi_1(\mathbb{C}^7\setminus{\Pi})\rightarrow\operatorname{Mod}_3^2$$ is a geometric homomorphism.  


If we glue a pair of open punctured disks to the boundary components of $\Sigma_3^2$ we obtain a punctured surface diffeomorphic to $\Sigma_{3,2}$. This procedure defines the capping homomorphism $\operatorname{Cap}:\operatorname{Mod}^2_3\rightarrow\operatorname{Mod}_{3,2}$ by extending the mapping classes in $\operatorname{Mod}_3^2$ to the be identity on the glued punctured disks. The proof of Theorem \ref{main} relies on the existence of a surjective homomorphism $$\theta:\pi_1(\CC^7\setminus\Pi)\rightarrow\pi_1^{orb}(\mathbb{P}\mathcal{H}(3,1))$$ such that the two monodromies $\rho:\pi_1(\CC^7\setminus\Pi)\rightarrow\operatorname{Mod}_3^2$ and $\rho_{\mathbb{P}\mathcal{H}(3,1)}:\pi_1^{orb}(\mathbb{P}\mathcal{H}(3,1))\rightarrow\operatorname{Mod}_{3,2}$ fit inside the following commutative diagram

\begin{center}
\begin{equation}
\label{comm}
\begin{tikzcd}
\pi_1(\CC^7\setminus\Pi) \arrow[r, "\theta"] \arrow[d, "\rho"]
& \pi_1^{orb}(\mathbb{P}\mathcal{H}(3,1)) \arrow[d, "\rho_{\mathbb{P}\mathcal{H}(3,1)}"] \\
\operatorname{Mod}_3^2 \arrow[r, "\operatorname{Cap}"]
& \operatorname{Mod}_{3,2}.
\end{tikzcd}  
\end{equation}
\end{center}

Let us define $\theta:\pi_1(\CC^7\setminus\Pi)\rightarrow\pi_1^{orb}(\mathbb{P}\mathcal{H}(3,1))$. We do so by composing two homomorphisms, where one of them has already been given in Proposition \ref{lo}. In what follows, we construct a surjective homomorphism $\theta_2:\pi_1(\CC^7\setminus\Pi)\rightarrow\pi_1^{orb}(\mathcal{M}^{\operatorname{flex}}_3)$. Then, the composition $\theta_1^{-1}\circ\theta_2$ will be the homomorphism $\theta$ we need in order to prove Theorem \ref{main}.

A pair of smooth quartics $Q_s$ and $Q_t$ might define the same isomorphism class of a Riemann surface. This is the case if and only if the parameters $s$ and $t$ are related by a weighted projective relation \cite[Proposition 1]{Shioda1993}. In particular, the vanishing loci $\mathbb{V}(Q_s)$ and $\mathbb{V}(Q_t)$ are isomorphic if and only if there exists $\lambda\in\CC^*$ such that 
\begin{equation}
\label{rel}
(s_1,s_2,s_3,s_4,s_5,s_6,s_7)=(\lambda t_1,\lambda^3t_2,\lambda^4t_3,\lambda^5t_4,\lambda^6t_5,\lambda^7t_6,\lambda^9t_7).
\end{equation}
The above relation is well-defined on $\Pi$. Indeed, the defining weighted polynomial of $\Pi$ has weights compatible with the weights of the relation in (\ref{rel}); we can see it by noticing that the weights given in (\ref{rel}) coincide with half the degrees $(d_1,\dots,d_7)$ of the homogeneous polynomials $q_1,\dots,q_7$ \cite[Section 3.7]{humphr}. In particular, the above relation is also well defined on $\CC^7\setminus\Pi$. 


Topologically, the weighted projective space obtained from the quotient of $\CC^7\setminus\Pi$ by the relation in (\ref{rel}) can be realized as a moduli space. In particular, it can be seen as the moduli space $\mathcal{M}^{\operatorname{flex}}_{3,\partial}$ of genus $3$ Riemann surfaces with $2$ boundary components and isomorphism classes given by the fibers of the surface bundle associated with $\CC^7\setminus\Pi$.

\begin{lemma}
\label{sper}
    The quotient map $l:\CC^7\setminus\Pi\rightarrow\mathcal{M}^{\operatorname{flex}}_{3,\partial}$ induces a surjective homomorphism $l_*:\pi_1(\CC^7\setminus\Pi)\rightarrow\pi_1(\mathcal{M}^{\operatorname{flex}}_{3,\partial})$ on the respective fundamental groups. 
\end{lemma}
\begin{proof}

The weighted projective relation defined in (\ref{rel}) on $\CC^7$ pulls back to a projective relation on the quotient $\CC^7/W(E_7)$ via the isomorphism $\tau:\CC^7/W(E_7)\rightarrow\CC^7$. In other words, the isomorphism $\tau$ induces a homeomorphism between the weighted projective space defined by (\ref{rel}) and $\mathbb{CP}^6$ modulo the induced linear action of $W(E_7)$.

Then, the quotient map $l$ can also be seen as the map
\begin{align*}
    l:\faktor{\CC^7\setminus\cup_{i\in I}H_i}{W(E_7)}\longrightarrow\faktor{\mathbb{P}(\CC^7\setminus\cup_{i\in I}H_i)}{W(E_7),}
\end{align*}
where $\mathbb{P}(\CC^7\setminus\cup_{i\in I}H_i)$ is the projectivization of the space $\CC^7\setminus\cup_{i\in I}H_i$. 

The map $l$ descends from the fiber bundle $\CC^7\setminus\cup_{i\in I}H_i\rightarrow\mathbb{P}(\CC^7\setminus\cup_{i\in I}H_i)$ via the free action of the finite group $W(E_7)$ and has connected fibers. In particular, the map $l$ is a fiber bundle with connected fibers and the induced homomorphism on the fundamental groups is surjective by applying the long exact sequence associated with $l$.




\end{proof}


The Teichm\"{u}ller cover of $\mathcal{M}^{\operatorname{flex}}_{3,\partial}$ will be denoted by $\mathcal{T}_{3,\partial}^{\operatorname{flex}}$ in the following proof.


\begin{proposition}
\label{li}
    Let $\rho^{\operatorname{flex}}:\pi_1^{orb}(\mathcal{M}_3^{\operatorname{flex}})\rightarrow\operatorname{Mod}_{3,2}$ be the monodromy of $\mathcal{M}_3^{\operatorname{flex}}$ and let $\rho:\pi_1(\CC^7\setminus\Pi)\rightarrow\operatorname{Mod}_3^2$ denote the monodromy of $\CC^7\setminus\Pi$.
    There exists a surjective homomorphism $\theta_2:\pi_1(\CC^7\setminus\Pi)\rightarrow\pi_1^{orb}(\mathcal{M}_3^{\operatorname{flex}})$ that commutes with the respective monodromies. In particular, the following diagram commutes

\begin{center}
\begin{equation*}
\label{comm}
\begin{tikzcd}
\pi_1(\CC^7\setminus\Pi) \arrow[r, "\theta_2"] \arrow[d, "\rho"]
& \pi_1^{orb}(\mathcal{M}_3^{\operatorname{flex}}) \arrow[d, "\rho^{\operatorname{flex}}"] \\
\operatorname{Mod}_3^2 \arrow[r, "\operatorname{Cap}"]
& \operatorname{Mod}_{3,2}.
\end{tikzcd}  
\end{equation*}
\end{center}
\end{proposition}

\begin{proof}

The surjective homomorphism $l_*:\pi_1(\CC^7\setminus\Pi)\rightarrow\pi_1(\mathcal{M}^{\operatorname{flex}}_{3,\partial})$ is induced by the quotient map, and in particular it is induced by a bundle map between surface bundles with isomorphic fibers. Therefore, the monodromies of $\CC^7\setminus\Pi$ and $\mathcal{M}^{\operatorname{flex}}_{3,\partial}$ must commute through $l_*$.

The group $\operatorname{Mod}_3^2$ is torsion-free and therefore the orbifold structure of $\mathcal{M}^{\operatorname{flex}}_{3,\partial}$ is not singular. In particular, the orbifold fundamental group of $\mathcal{M}^{\operatorname{flex}}_{3,\partial}$ can be identified with its fundamental group $\pi_1(\mathcal{M}_{3,\partial}^{\operatorname{flex}})$. In other words, if $\rho^{\operatorname{flex}}_{\partial}$ is the monodromy of the moduli space $\mathcal{M}^{\operatorname{flex}}_{3,\partial}$, the diagram

\[
    \begin{tikzcd}[column sep=1em]
\pi_1(\CC^7\setminus\Pi) \arrow[rr, "l_*"] \arrow[dr, "\rho"']& & \pi_1(\mathcal{M}_{3,\partial}^{\operatorname{flex}})\arrow[ld, "\rho^{\operatorname{flex}}_{\partial}"]\\
 & \operatorname{Mod}_3^2
\end{tikzcd}
    \]

must commutes.

Suppose now $\mathcal{T}_{3,\partial}^{\operatorname{flex}}$ and $\mathcal{T}_3^{\operatorname{flex}}$ are the Teichm\"{u}ller covers of $\mathcal{M}^{\operatorname{flex}}_{3,\partial}$ and $\mathcal{M}^{\operatorname{flex}}_3$, respectively. There exists a map $\mathcal{T}_{3,\partial}^{\operatorname{flex}}\rightarrow\mathcal{T}_3^{\operatorname{flex}}$ given by collapsing the lengths of the boundary components to zero. In particular, this map is the restriction of the classic projection given on the respective global Teichm\"{u}ller spaces where the preimage of $\mathcal{T}_3^{\operatorname{flex}}$ is exactly $\mathcal{T}_{3,\partial}^{\operatorname{flex}}$. Hence, the induced map $\pi_1(\mathcal{T}_{3,\partial}^{\operatorname{flex}})\rightarrow\pi_1(\mathcal{T}_{3}^{\operatorname{flex}})$ on the fundamental groups is surjective.

Consider the images of $\rho^{\operatorname{flex}}_{\partial}$ in $\operatorname{Mod}_{3}^2$ and of $\rho^{\operatorname{flex}}$ in $\operatorname{Mod}_{3,2}$. Every marking of a Riemann surface in $\mathcal{M}^{\operatorname{flex}}_3$ appears as the image of a marking associated with a Riemann surface in $\mathcal{M}^{\operatorname{flex}}_{3,\partial}$. Then, the restriction of the homomorphism $\operatorname{Cap}:\operatorname{Mod}_3^2\rightarrow\operatorname{Mod}_{3,2}$ on $\operatorname{im}\rho^{\operatorname{flex}}_{\partial}$ is surjective onto $\operatorname{im}\rho^{\operatorname{flex}}$ and the homomorphism $\pi_1^{orb}(\mathcal{M}^{\operatorname{flex}}_{3,\partial})\rightarrow\pi_1^{orb}(\mathcal{M}^{\operatorname{flex}}_3)$ must be surjective too.


\end{proof}

Our final goal is to prove Theorem \ref{main}. In particular, we will show that the monodromy $\rho_{\mathbb{P}\mathcal{H}(3,1)}$ is geometric. We are going to prove Theorem \ref{main} using the following lemma.

\begin{lemma}
\label{saveass}
Every surjective endomorphism of $A(E_7)_\Delta$ is an isomorphism.
\end{lemma}
\begin{proof}
Both $A(E_7)$ and the automorphism group $\operatorname{Aut}(A(E_7))$ are residually finite \cite[Theorem 1]{BaumslagResidually1963} because $A(E_7)$ is linear \cite{Cohen2002}. Hence, the inner subgroup $A(E_7)_\Delta$ of $\operatorname{Aut}(A(E_7))$ is both finitely generated and residually finite. In particular, we can conclude that every surjective endomorphism of $A(E_7)_\Delta$ is an isomorphism \cite[Chapter III, Proposition 7.5]{metricbridson}.

\end{proof}

\begin{proof}[Proof of Theorem \ref{main}] Since $\CC^7\setminus\Pi$ is an Eilenberg-Maclane space for the Artin group $A(E_7)$, the fundamental group $\pi_1(\CC^7\setminus\Pi)$ is isomorphic to $A(E_7)$. Moreover, from Theorem \ref{lomo} we know that $\pi_1^{orb}(\mathbb{P}\mathcal{H}(3,1))$ is isomorphic to $A(E_7)_\Delta$. 

Let us consider the surjective homomorphism $\theta:\pi_1(\CC^7\setminus\Pi)\rightarrow\pi_1^{orb}(\mathbb{P}\mathcal{H}(3,1))$ as the composition $$\pi_1(\CC^7\setminus\Pi)\xrightarrow{\theta_2}\pi^{orb}(\mathcal{M}^{\operatorname{flex}}_3)\xrightarrow{\theta_1^{-1}}\pi_1^{orb}(\mathbb{P}\mathcal{H}(3,1)).$$

The group $A(E_7)_\Delta$ is centerless. Therefore, the kernel of the homomorphism $\theta:A(E_7)\rightarrow A(E_7)_\Delta$ contains the subgroup $\langle\Delta\rangle$. This implies that the induced map $$\overline{\theta}:A(E_7)_\Delta\rightarrow A(E_7)_\Delta$$
is a well-defined surjective endomorphism of $A(E_7)_\Delta$ and therefore an isomorphism by Lemma \ref{saveass}. 

Since $\theta$ commutes with the monodromies $\rho:\pi_1(\CC^7\setminus\Pi)\rightarrow\operatorname{Mod}_3^2$ and $\rho_{\mathbb{P}\mathcal{H}(3,1)}:\pi_1^{orb}(\mathbb{P}\mathcal{H}(3,1))\rightarrow\operatorname{Mod}_{3,2}$ through the capping homomorphism $\operatorname{Cap}:\operatorname{Mod}_3^2\rightarrow\operatorname{Mod}_{3,2}$, we can conclude that $\rho_{\mathbb{P}\mathcal{H}(3,1)}$ is geometric since $\rho$ is.

\end{proof}


\Addresses

\end{document}